\documentclass[a4paper,
              twoside,
              ]{amsart}
\usepackage{amssymb} 
\usepackage{MnSymbol} 
\usepackage{mathtools} 
\usepackage[all]{xy} 
\usepackage{color} 
\usepackage{ifdraft} 
\usepackage{lastpage} 
\usepackage[pdfstartview=FitH]{hyperref} 

\DeclareMathAlphabet{\mathpzc}{OT1}{pzc}{m}{it}

\newtheorem{thm}{Theorem}[section]

\newtheorem{cor}[thm]{Corollary}
\newtheorem{lem}[thm]{Lemma}

\newtheorem{prop}[thm]{Proposition}

\theoremstyle{definition}
\newtheorem{defn}[thm]{Definition}

\theoremstyle{remark}
\newtheorem{rem}[thm]{Remark}

\newtheorem{exmpl}[thm]{Example}

\numberwithin{equation}{section}

\providecommand\given{}
\newcommand\SetSymbol[1][]{%
\nonscript\:#1\vert
\allowbreak
\nonscript\:
\mathopen{}}
\DeclarePairedDelimiterX\set[1]\{\}{%
\renewcommand\given{\SetSymbol[\delimsize]}
#1
}

\newcommand*{\cat}[1]{\mathbf{#1}} 
\newcommand*{\mto}{\rightarrow} 
\newcommand*{\wto}{\xrightarrow{\sim}} 
\newcommand*{\cg}[1]{\prescript{\sharp}{}{\!#1}} 
\newcommand*{\bh}{d} 
\newcommand*{\Int}{\mathbb{Z}} 
\newcommand*{\Rat}{\mathbb{Q}} 
\newcommand*{\Real}{\mathbb{R}} 
\newcommand*{\FF}{\mathbb{F}} 
\newcommand*{\cmplx}[1]{{#1}^\bullet} 
\newcommand*{\tensor}{\otimes} 
\newcommand*{\ctensor}{\hat{\otimes}} 
\newcommand*{\Ltensor}{\tensor^{\mathbb{L}}} 
\newcommand*{\isomorph}{\cong} 
\newcommand*{\op}{\mathrm{op}} 
\newcommand*{\cont}{\mathrm{cont}} 
\newcommand{\cyc}{\mathrm{cyc}} 
\newcommand*{\cycchar}{\varepsilon_\mathrm{cyc}} 
\newcommand*{\algc}[1]{\overline{#1}} 
\newcommand*{\sheaf}[1]{\mathpzc{#1}} 
\newcommand{\openideals}{\mathfrak{I}} 
\newcommand{\Sect}{\Gamma} 
\newcommand{\Val}{\mathcal{O}} 
\newcommand{\ncL}{\mathcal{L}} 
\newcommand{\X}{\mathcal{X}} 
\newcommand{\decomp}{\mathcal{D}} 
\newcommand{\inertia}{\mathcal{I}} 
\newcommand{\kinert}{\mathcal{K}} 
\newcommand{\ginert}{\mathcal{J}} 
\newcommand*{\dual}[1]{{#1}^{\vee}} 
\newcommand{\rquot}{\backslash} 
\newcommand{\Motive}{\mathcal{M}} 
\newcommand{\Gm}{{\mathbb{G}_{\mathrm{m}}}} 
\newcommand{\mdual}{\ast} 
\newcommand{\Mdual}{\circledast} 
\newcommand{\epsfact}{\varepsilon} 

\newcommand{\ringtransf}{\Psi} 
\newcommand{\id}{\mathrm{id}} 
\newcommand{\Frob}{\mathfrak{F}} 
\newcommand{\eval}{\Phi} 

\DeclareMathOperator{\Cone}{Cone} 
\DeclareMathOperator{\HF}{H} 
\DeclareMathOperator{\Hom}{Hom} 
\DeclareMathOperator{\Ext}{Ext} 
\DeclareMathOperator{\Tor}{Tor} 
\DeclareMathOperator{\Jac}{Jac} 
\DeclareMathOperator{\Aut}{Aut} 
\DeclareMathOperator{\sheafHom}{\mathpzc{Hom}} 
\DeclareMathOperator{\coker}{coker} 
\DeclareMathOperator{\Spec}{Spec} 
\DeclareMathOperator{\RDer}{R} 
\DeclareMathOperator{\Gal}{Gal} 
\DeclareMathOperator{\KTh}{K} 
\DeclareMathOperator{\Sel}{Sel} 
\DeclareMathOperator{\Tate}{T} 
\DeclareMathOperator{\Pic}{Pic} 
\DeclareMathOperator{\Div}{\sheaf{Div}} 
\DeclareMathOperator{\SC}{SC} 
\DeclareMathOperator{\GL}{GL} 
\DeclareMathOperator{\SL}{SL} 

\newcommand{\comment}[1]{\ifdraft{\textcolor{red}{#1}}{}}

\newdir{ >}{{}*!/-5pt/@{>}} 
\newdir{O}{{}*-\cir<2.5pt>{}} 
\xyoption{2cell} 
\UseTwocells %
\xyoption{poly}

\allowdisplaybreaks[1]

\begin{document}

\title[Noncommutative Iwasawa Main Conjecture]{On a Noncommutative Iwasawa Main Conjecture for Function Fields}
\author{Malte Witte}%

\address{Malte Witte\newline Ruprecht-Karls-Universit\"at Heidelberg\newline
Mathematisches Institut\newline
Im Neuenheimer Feld 288\newline
D-69120 Heidelberg }%
\email{witte@mathi.uni-heidelberg.de}


\date{\today}%

\begin{abstract}
We formulate and prove an analogue of the non-commutative Iwasawa Main Conjecture for $\ell$-adic representations of the Galois group of a function field of characteristic $p$. We also prove a functional equation for the resulting non-commutative $L$-functions. As corollaries, we obtain non-commutative generalisations of the main conjecture for Picard-$1$-motives of Greither and Popescu and a main conjecture for abelian varieties over function fields in precise analogy to the $\GL_2$ main conjecture of Coates, Fukaya, Kato, Sujatha and Venjakob.
\end{abstract}

\maketitle

\section{Introduction}

\noindent In \cite{CFKSV}, Coates, Fukaya, Kato, Sujatha and Venjakob formulate a non-commu\-tative Iwasawa Main Conjecture for $\ell$-adic Lie extensions of number fields. Other, partly more general versions are formulated in \cite{HK:EBKC+IMC}, \cite{RW:EquivIwaTh2}, and \cite{FK:CNCIT}. A geometric version for separated schemes of finite type over a finite field is formulated and proved in \cite{Witte:MCVarFF}, \cite{Burns:MCinGIwTh+RelConj}, and \cite{Witte:UnitLFunctions}.

Building on the ideas of \cite{Witte:MCVarFF} we will consider in this article  admissible $\ell$-adic Lie extensions $K_{\infty}/K$ of a function field $K$ of transcendence degree $1$ over a finite field $\FF$ of characteristic $p$. Here, an extension $K_\infty/K$ inside a fixed separable closure of $K$ is called an \emph{admissible $\ell$-adic Lie extension} if
\begin{enumerate}
 \item $K_\infty/K$ is Galois and the Galois group $G\coloneqq\Gal(K_\infty/K)$ is an $\ell$-adic Lie group,
 \item $K_\infty/K$ contains the unique $\Int_\ell$-extension of $\FF$,
 \item $K_\infty/K$ is unramified outside a finite set of places.
\end{enumerate}

We will formulate and prove a non-commutative main conjecture for the extension $K_\infty/K$ and any continuous representation $T$ of the absolute Galois group of $K$ which is unramified outside a finite set of primes (Thm.~\ref{thm:MCforRep l p different}, Thm.~\ref{thm:MCforRep l equal p myLfunc}). We will also prove a functional equation for the corresponding non-commutative $L$-functions (Thm.~\ref{thm:functional equation}). From this main conjecture, we will deduce a main conjecture for Greenberg's Selmer group of $T$ (Cor.~\ref{cor:mc for selmer groups}), non-commutative generalisations of the main conjecture for the $\ell$-adic Tate module of a Picard-$1$-motive from \cite{GreitherPopescu:PicardOneMotives} (Cor.~\ref{cor:picard motives Tate twist 1}, Cor.~\ref{cor:Picard motives Tate twist 0}), a non-commutative generalisation of the classical main conjecture for the Galois group of the maximal abelian $\ell$-extension of $K_\infty$ unramified outside a finite set of places (Cor.~\ref{cor:classical mc}), and an analogue for abelian varieties over function fields of the $\operatorname{Gl}_2$ main conjecture in \cite{CFKSV} for $\ell\neq p$ (Cor.~\ref{cor:MCforAbVars}).

Let $\Int_{\ell}[[G]]$ be the Iwasawa algebra of $G$ and let $\Int_{\ell}[[G]]_S$ be its localisation at Venjakob's canonical Ore set $S$.  Let $C$ be the smooth and proper curve corresponding to $K$ and assume for simplicity that $C$ is geometrically connected. We also fix two disjoint, possibly empty sets $\Sigma$ and $\Sigma'$ of closed points in $C$. To the representation $T$ we will associate a certain Selmer complex $\cmplx{\SC}_{\Sigma,\Sigma'}(T)$. The non-commutative main conjecture claims that $\cmplx{\SC}_{\Sigma,\Sigma'}(T)$  is a perfect complex of $\Int_{\ell}[[G]]$-modules which is $S$-torsion and hence, gives rise to a class $[\cmplx{\SC}_{\Sigma,\Sigma'}(T)]$ in the relative $\KTh$-group $\KTh_0(\Int_{\ell}[[G]],S)$. Moreover, it postulates the existence of a non-commutative $L$-function $\ncL_{\Sigma,\Sigma'}(T)\in \KTh_1(\Int_{\ell}[[G]]_S)$ which maps to $-[\cmplx{\SC}_{\Sigma,\Sigma'}(T)]$ under the connecting homomorphism
$$
\bh\colon \KTh_1(\Int_{\ell}[[G]]_S)\mto \KTh_0(\Int_{\ell}[[G]],S)
$$
and which satisfies an interpolation property with respect to values of the $\Sigma$-truncated $\Sigma'$-modified $L$-function of the twist $T(\rho)$ of $T$ by representations $\rho$ of $G$.

The functional equation states that
\[
 \ncL_{\Sigma',\Sigma}(T^{\mdual})^{\Mdual}\epsfact_{\Sigma,\Sigma'}(T)\ncL_{\Sigma,\Sigma'}(T)=1,
\]
where $T^{\mdual}$ is the dual representation, $\epsfact_{\Sigma,\Sigma'}(T)$ is the global $\epsfact$-factor, and
\[
 \Mdual\colon\KTh_1(\Int_{\ell}[[G]]_S)\mto\KTh_1(\Int_{\ell}[[G]]_S)
\]
is induced by sending an invertible matrix to the transposed of its inverse.

Usually, one assumes that $\Sigma$ contains the ramification locus of $K_\infty/K$ and of $T$, but we will show that this condition can be weakened. Moreover, our main conjecture also applies to Iwasawa algebras over much more general coefficient rings, for example, the power series rings $\Int_\ell[[X_1,\dots,X_n]]$ and their factor rings as well as rings that are itself Iwasawa algebras of $\ell$-adic Lie groups.

Greenberg's Selmer group may be identified with the Pontryagin dual of the second cohomology group of the complex $\cmplx{\SC}_{\Sigma,\emptyset}(T)$. The generalisation of the main conjecture for Picard-$1$-motives of Greither and Popescu follows from the special cases $T=\Int_\ell$ and $T=\Int_\ell(1)$ with $\Sigma$ and $\Sigma'$ non-empty, the generalisation of the classical main conjecture follows from the case $T=\Int_\ell(1)$ and $\Sigma'=\emptyset$. For the non-commutative main conjecture for abelian varieties one chooses $T$ to be the Tate module of the dual abelian variety.

The article is structured as follows. In Section~\ref{sec:Preliminaries} we recall the $\KTh$-theoretic framework. As in  \cite{Witte:MCVarFF} we will use the language of Waldhausen categories as a convenient tool to handle the relevant elements in $\KTh$-groups. In particular, we recall the construction of the Waldhausen category of perfect complexes of adic sheaves and the construction of the total derived section functor as a Waldhausen exact functor.
In Section~\ref{sec:Selmer Complexes}, this will be used to construct the Selmer complex $\cmplx{\SC}_{\Sigma,\Sigma'}(T)$ and to analyse its properties. Section~\ref{sec:nc main conjectures} contains the definition of the $\Sigma$-truncated $\Sigma'$-modified $L$-function as well as the formulation and proofs of the different main conjectures, including the description of the cohomology of $\cmplx{\SC}_{\Sigma,\Sigma'}(T)$.

Large parts of this article have been inspired by \cite{FK:CNCIT}. There has been previous work on a non-commutative main conjecture for elliptic curves over function fields in the case $\ell\neq p$ \cite{Sechi:IMCoverGlobalFunctionFields} and independent work on the $\cat{M}_H(G)$-conjecture for $\ell$-adic Selmer groups of abelian varieties over function fields \cite{BandiniValentino:ControlTheorems}. We do not treat non-commutative main conjectures for abelian varieties in the case $\ell=p$, which necessitates different methods. In the recent preprint \cite{VauclairTrihan:NCITAVFF}, Trihan and Vauclair announce a proof of the conjectue in this case.

This article is an extended and improved version of a preprint that has been circulated since 2013. The author wants to thank A. Huber for suggesting to work on this question and A. Schmidt as well as O. Venjakob for valuable discussions.

\section{Preliminaries}\label{sec:Preliminaries}

\subsection{The \texorpdfstring{$\KTh$}{K}-theoretic framework}\label{ss:Framework}

We will recall some $\KTh$-theoretic constructions from \cite{Witte:MCVarFF}. All rings in this article will be associative with identity; a module over a ring will always refer to a left unitary module. For a ring $R$, $R^\op$ will denote the opposite ring.
Recall that an \emph{adic ring} is a ring  $\Lambda$ such that for each $n\geq 1$ the $n$-th power of the Jacobson radical $\Jac(\Lambda)^n$ is of finite index in $\Lambda$ and
$$
\Lambda=\varprojlim_{n\geq 1}\Lambda/\Jac(\Lambda)^n.
$$
In particular, $\Lambda$ carries a natural profinite topology. Note that a commutative adic ring is always noetherian and semi-local \cite[Cor. 36.35]{Warner:TopRings}.

Write $\openideals_{\Lambda}$ for the set of open two-sided ideals of $\Lambda$, partially ordered by inclusion. To compute the Quillen $\KTh$-groups of $\Lambda$ one can apply the Waldhausen $S$-construction to one of a variety of different Waldhausen categories associated to $\Lambda$. As in \cite{Witte:MCVarFF}, we will use the following, which turns out to be particularly convenient for our purposes.

We recall that for any ring $R$, a complex $\cmplx{M}$ of $R$-modules is called \emph{$DG$-flat} if every module $M^n$ is flat and for every acyclic complex $\cmplx{N}$ of $R^\op$-modules, the total complex $\cmplx{(N\tensor_R M)}$ is acyclic. The complex $\cmplx{M}$ is called \emph{perfect} if it is quasi-isomorphic to a complex $\cmplx{P}$ such that $P^n$ is finitely generated and projective and $P^n=0$ for almost all $n$.

\begin{defn}\label{defn:PDG(Lambda)}
Let $\Lambda$ be an adic ring. We denote by
$\cat{PDG}^{\cont}(\Lambda)$ the following Waldhausen category.
The objects of $\cat{PDG}^{\cont}(\Lambda)$ are inverse system
$(\cmplx{P}_I)_{I\in \openideals_{\Lambda}}$ satisfying the
following conditions:
\begin{enumerate}
\item for each $I\in\openideals_{\Lambda}$, $\cmplx{P}_I$ is a
$DG$-flat perfect complex of $\Lambda/I$-modules,

\item for each $I\subset J\in\openideals_{\Lambda}$, the
transition morphism of the system
$$
\varphi_{IJ}\colon\cmplx{P}_I\mto \cmplx{P}_J
$$
induces an isomorphism
$$
\Lambda/J\tensor_{\Lambda/I}\cmplx{P}_I\isomorph \cmplx{P}_J.
$$
\end{enumerate}
A morphism of inverse systems $(f_I\colon \cmplx{P}_I\mto
\cmplx{Q}_I)_{I\in\openideals_{\Lambda}}$ in
$\cat{PDG}^{\cont}(\Lambda)$ is a weak equivalence if every $f_I$
is a quasi-isomorphism. It is a cofibration if every $f_I$ is
injective and the system $(\coker f_I)$ is in $\cat{PDG}^{\cont}(\Lambda)$.
\end{defn}

By \cite[Prop.~3.7]{Witte:MCVarFF} the complex
$$
\varprojlim_{I\in\openideals_{\Lambda}}\cmplx{Q}_I
$$
is a perfect complex for every system $(\cmplx{Q}_I)_{I\in\openideals_{\Lambda}}$ in $\cat{PDG}^{\cont}(\Lambda)$ and the $\KTh$-groups $\KTh_n(\cat{PDG}^{\cont}(\Lambda))$ of the Waldhausen category coincide with the Quillen $\KTh$-groups $\KTh_n(\Lambda)$ of the adic ring $\Lambda$.

If $\Lambda'$ is a second adic ring and $M$ is a $\Lambda'$-$\Lambda$-bimodule which is finitely generated and projective as $\Lambda'$-module, then the derived tensor product with $M$ induces natural homomorphisms $\KTh_{n}(\Lambda)\mto \KTh_{n}(\Lambda')$. It also follows from \cite[Prop.~3.7]{Witte:MCVarFF} that these homomorphisms coincide with the homomorphisms induced from the following Waldhausen exact functor.

\begin{defn}\label{defn:change of ring functor}
For $(\cmplx{P}_I)_{I\in\openideals_{\Lambda}}\in
\cat{PDG}^{\cont}(\Lambda)$
we define a Waldhausen exact functor
$$
\ringtransf_{M}\colon
\cat{PDG}^{\cont}(\Lambda)\mto\cat{PDG}^{\cont}(\Lambda'),\qquad \cmplx{P}\mto(\varprojlim_{J\in\openideals_{\Lambda}}
\Lambda'/I\tensor_{\Lambda'}\cmplx{(M\tensor_{\Lambda}P_{J})})_{I\in\openideals_{\Lambda'}}.
$$
\end{defn}

Let now $G$ be a profinite group which is the semi-direct product $H\rtimes\Gamma$ of a closed normal subgroup $H$ and a closed subgroup $\Gamma$ which isomorphic to $\Int_{\ell}$. We assume that $H$ is a topologically finitely generated and a \emph{virtual pro-$\ell$-group}, i.\,e.\ it contains an open pro-$\ell$-group. If $\Lambda$ is an adic $\Int_{\ell}$-algebra, then so are the profinite group rings $\Lambda[[G]]$ and $\Lambda[[H]]$ \cite[Prop. 3.2]{Witte:MCVarFF}.

Assume in addition that $\Lambda[[H]]$ is noetherian. Then the set
$$
S\coloneqq\set{f\in \Lambda[[G]]\given \text{$\Lambda[[G]]/\Lambda[[G]]f$ is finitely generated as $\Lambda[[H]]$-module}}
$$
is a left and right denominator set in $\Lambda[[G]]$ such that the localisation $\Lambda[[G]]_S$ exists \cite[Lemma 3.6]{Witte:BCP}. Moreover, the $\KTh$-theory localisation sequence for $S\subset \Lambda[[G]]$ splits into short split exact sequences \cite[Cor.~3.3]{Witte:Splitting}. In particular, we obtain a split exact sequence
$$
0\mto \KTh_1(\Lambda[[G]])\mto \KTh_1(\Lambda[[G]]_S)\xrightarrow{\bh} \KTh_0(\Lambda[[G]],S)\mto 0.
$$

We will describe this sequence in more detail, introducing the following Waldhausen categories.

\begin{defn}\label{defn:wHPDG}
Let $\Lambda$ be any adic $\Int_\ell$-algebra and $G=H\rtimes\Gamma$ with $\Gamma\isomorph\Int_\ell$ and $H$ a topologically finitely generated virtual pro-$\ell$-group. We write $\cat{PDG}^{\cont,w_H}(\Lambda[[G]])$ for the full
Waldhausen subcategory of $\cat{PDG}^{\cont}(\Lambda[[G]])$ of
objects $(\cmplx{P}_J)_{J\in\openideals_{\Lambda[[G]]}}$ such that
$$
\varprojlim_{J\in\openideals_{\Lambda[[G]]}} \cmplx{P}_J
$$
is a perfect complex of $\Lambda[[H]]$-modules.

We write $w_H\cat{PDG}^{\cont}(\Lambda[[G]])$ for the Waldhausen
category with the same objects, morphisms and cofibrations as
$\cat{PDG}^{\cont}(\Lambda[[G]])$, but with a new set of weak
equivalences given by those morphisms whose cones are objects of the category
$\cat{PDG}^{\cont,w_H}(\Lambda[[G]])$.
\end{defn}

If $\Lambda[[H]]$ is noetherian, we may then identify for all $n\geq 0$
\begin{equation}\label{eqn:def of K-groups}
\begin{aligned}
\KTh_{n}(\Lambda[[G]],S)&=\KTh_{n}(\cat{PDG}^{\cont,w_H}(\Lambda[[G]])),\\
\KTh_{n+1}(\Lambda[[G]]_S)&=\KTh_{n+1}(w_H\cat{PDG}^{\cont}(\Lambda[[G]]))
\end{aligned}
\end{equation}
\cite[\S~4]{Witte:MCVarFF}. For example, $\Lambda[[H]]$ is noetherian if $\Lambda$ is a commutative adic $\Int_\ell$-algebra and $H$ is a compact $\ell$-adic Lie group \cite[Cor.~3.4]{Witte:Splitting}.

If $\Lambda[[H]]$ is not noetherian, then $S$ fails to be a left or right denominator set in general \cite[Ex.~3.7]{Witte:BCP}. We will then take \eqref{eqn:def of K-groups} as a sensible definition such that the splitting result \cite[Cor.~3.3]{Witte:Splitting} remains true.

In particular, $\KTh_{0}(\Lambda[[G]],S)$ is the abelian group generated by the symbols $[\cmplx{P}]$ with $\cmplx{P}$ an object in $\cat{PDG}^{\cont,w_H}(\Lambda[[G]])$ modulo the relations
\begin{align*}
[\cmplx{P}]&=[\cmplx{Q}]&&\text{if $\cmplx{P}$ and $\cmplx{Q}$ are quasi-isomorphic,}\\
[\cmplx{P}_2]&=[\cmplx{P}_1]+[\cmplx{P}_3]&&\text{if $0\mto\cmplx{P}_1\mto\cmplx{P}_2\mto\cmplx{P}_3\mto 0$ is an exact sequence.}
\end{align*}
The groups $\KTh_1(\Lambda[[G]])$ and $\KTh_1(\Lambda[[G]]_S)$ are generated by the symbols $[f\lcirclearrowright \cmplx{P}]$ where $f\colon \cmplx{P}\mto \cmplx{P}$ is an endomorphism that is a weak equivalence in the Waldhausen categories $\cat{PDG}^{\cont}(\Lambda[[G]])$ and $w_H\cat{PDG}^{\cont}(\Lambda[[G]])$, respectively \cite[Prop.~3.10]{Witte:MCVarFF}. The map $\KTh_1(\Lambda[[G]])\mto \KTh_1(\Lambda[[G]]_S)$ is the obvious one; the boundary map
$$
\bh\colon\KTh_1(\Lambda[[G]]_S)\mto\KTh_0(\Lambda[[G]],S)
$$
is given by
$$
\bh[f\lcirclearrowright \cmplx{P}]=-[\cmplx{\Cone(f)}]
$$
where $\cmplx{\Cone(f)}$ denotes the cone of $f$ \cite[Thm.~A.5]{Witte:MCVarFF}. (We note that other authors use $-\bh$ instead.)
\begin{rem}\label{rem:class of a projective RH module}
Let $M$ be an $\Lambda[[G]]$-module which has a resolution by a strictly perfect complex of $\Lambda[[H]]$-modules $\cmplx{Q}$. By \cite[Lemma 3.11]{Witte:BCP}, $M$ then also has a resolution by a strictly perfect complex $\cmplx{P}$ of $\Lambda[[G]]$-modules. We may identify $\cmplx{P}$ with the object
$$
(\Lambda[[G]]/J\tensor_{\Lambda[[G]]}\cmplx{P})_{J\in\openideals_{\Lambda[[G]]}}
$$
in $\cat{PDG}^{\cont, w_H}(\Lambda[[G]])$ and set
$$
[M]\coloneqq[\cmplx{P}]\in\KTh_0(\Lambda[[G]],S).
$$
Note that $[M]$ does not depend on the particular choice of the resolutions $\cmplx{P}$ or $\cmplx{Q}$.
\end{rem}

Fix a second adic $\Int_\ell$-algebra $\Lambda'$ and let $\rho$ be a $\Lambda'$-$\Lambda[[G]]$-bimodule which is finitely generated and projective as $\Lambda'$-module. Consider
\[
E_\rho\coloneqq \Lambda'[[\Gamma]]\tensor_{\Lambda'}\rho
\]
as $\Lambda'[[\Gamma]]$-$\Lambda[[G]]$-bimodule with the left $\Gamma$-operation on the first factor $\Lambda'[[\Gamma]]$ and the right diagonal operation of $g\in G$ by its image in $\Gamma$ on $\Lambda'[[\Gamma]]$ and by $g$ on $\rho$. Note that
\[
\ringtransf_{E_\rho}\colon\cat{PDG}^{\cont}(\Lambda[[G]])\mto\cat{PDG}^{\cont}(\Lambda'[[\Gamma]])
\]
takes objects of $\cat{PDG}^{\cont,w_H}(\Lambda[[G]])$ to objects of $w_{1}\cat{PDG}^{\cont,w_1}(\Lambda'[[\Gamma]])$ , where $1$ denotes the trivial subgroup of $\Gamma$ \cite[Prop.~4.6]{Witte:MCVarFF}.

We define the evaluation map
\begin{equation}\label{eq:def of evaluation map}
\eval_\rho\colon \KTh_1(\Lambda[[G]]_S)\mto\KTh_1(\Lambda'[[\Gamma]]_S)
\end{equation}
as the map induced by $\ringtransf_{E_\rho}$ on the $\KTh$-groups. If $\Lambda'=\Lambda=\Int_\ell$, then $\rho$ corresponds to a representation $\rho^\sharp$ of $G$ over $\Int_\ell$ by considering the left action of $g\in G$ on $\rho$ by $g^{-1}$. Beware that we deviate from the sign convention used in \cite[(22)]{CFKSV}. In terms of the cited article, our $\eval_\rho$ corresponds to the evaluation at the representation dual to $\rho^\sharp$.

Assume that $R$ is a commutative, local, and regular adic $\Int_\ell$-algebra. By the Cohen Structure Theorem \cite[Ch. VIII, \S 5, Thm.~2]{Bourbaki:CommAlg}, we have
\[
R\isomorph R_0[[X_1,\dots,X_n]],
\]
where $R_0$ is either a finite field of characteristic $\ell$ or the valuation ring of a finite field extension of $\Rat_\ell$ and $X_1,\dots, X_n$ are indeterminates. In particular, we may identify $R$ with the profinite group algebra of $\Int_{\ell}^n$ with coefficients in $R_0$.

If $G=H\rtimes\Gamma$ is a $\ell$-adic Lie group without elements of order $\ell$, then the rings $R[[G]]$ and $R[[H]]$ are both noetherian and of finite global dimension \cite[Thm.~4.1]{Brumer:PseudocompactAlgebras}. Let $\cat{N}_H(R[[G]])$ denote the abelian category of finitely generated $R[[G]]$-modules which are also finitely generated as $R[[H]]$-modules. Note that
\begin{equation}\label{eq:NHG computes relative K-group}
\begin{split}
\KTh_0(\cat{PDG}^{\cont,w_H}(R[[G]]))\mto \KTh_0(\cat{N}_H(R[[G]])),\\
[(\cmplx{P}_I)_{I\in\openideals_{R[[G]]}}]\mapsto \sum_{i=-\infty}^{\infty} (-1)^i[\HF^i(\varprojlim_{I\in\openideals_{R[[G]]}}\cmplx{P}_I)]
\end{split}
\end{equation}
is an isomorphism. The inverse is given by the construction in Rem.~\ref{rem:class of a projective RH module}.

If the quotient field of $R$ is of characteristic $0$, one may also consider the abelian category $\cat{M}_H(R[[G]])$ of finitely generated $R[[G]]$-modules whose $\ell$-torsion-free part is finitely generated as $R[[H]]$-module and the left denominator set
$$
S^*\coloneqq\bigcup_{n}\ell^n S\subset R[[G]].
$$
Still assuming that $G$ has no element of order $\ell$ it is known that the natural maps
$$
\KTh_1(R[[G]]_{S})\mto\KTh_1(R[[G]]_{S^*}),\qquad \KTh_0(\cat{N}_H(R[[G]]))\mto\KTh_0(\cat{M}_H(R[[G]]))
$$
are split injective \cite[Prop.~3.4]{BV:DescentTheory} and fit into a commutative diagram
$$
\xymatrix{
0\ar[r]&\KTh_1(R[[G]])\ar[r]\ar[d]^{=}&\KTh_1(R[[G]]_S)\ar[r]^-\bh\ar[d]&\KTh_0(\cat{N}_H(R[[G]]))\ar[r]\ar[d]&0\\
0\ar[r]&\KTh_1(R[[G]])\ar[r]&\KTh_1(R[[G]]_{S^*})\ar[r]^-\bh&\KTh_0(\cat{M}_H(R[[G]]))\ar[r]&0
}
$$
In particular, an identity of the type $f=\bh g$ in $\KTh_0(\cat{N}_H(R[[G]]))$ will imply a corresponding identity in $\KTh_0(\cat{M}_H(R[[G]]))$. It is $\cat{M}_H(R[[G]])$ which plays a central role in the original formulation of the non-commutative Iwasawa Main Conjecture \cite{CFKSV}. However, we will not make use of $\cat{M}_H(R[[G]])$ in the following.

\subsection{Perfect complexes of adic sheaves}\label{ss:perfect complexes of adic sheaves}

In this section we recall some constructions from \cite[\S~5.4--5.5]{Witte:PhD}.

Let $X$ be a separated scheme of finite type over a finite field $\FF$ (We will only need the case $\dim X\leq 1$). Recall that for a finite ring $R$, a complex $\cmplx{\sheaf{F}}$ of \'etale sheaves of left $R$-modules on $X$ is called \emph{strictly perfect} if it is strictly bounded and each $\sheaf{F}^n$ is constructible and flat. It is \emph{perfect} if it is quasi-isomorphic to a strictly perfect complex. We call it \emph{$DG$-flat} if for each geometric point of $X$, the complex of stalks is $DG$-flat.

\begin{defn}\label{defn:PDGcont(X,Lambda)}
Let $X$ be a separated scheme of finite type over a finite field and let $\Lambda$ be an
adic ring. The \emph{category of perfect complexes of adic
sheaves} $\cat{PDG}^{\cont}(X,\Lambda)$ is the following
Waldhausen category. The objects of $\cat{PDG}^{\cont}(X,\Lambda)$
are inverse systems $(\cmplx{\sheaf{F}}_I)_{I\in
\openideals_{\Lambda}}$ such that:
\begin{enumerate}
\item for each $I\in\openideals_{\Lambda}$, $\cmplx{\sheaf{F}}_I$
is $DG$-flat perfect complex of \'etale sheaves of $\Lambda/I$-modules on $X$,

\item for each $I\subset J\in\openideals_{\Lambda}$, the
transition morphism
$$
\varphi_{IJ}\colon\cmplx{\sheaf{F}}_I\mto \cmplx{\sheaf{F}}_J
$$
of the system induces an isomorphism
$$
\Lambda/J\tensor_{\Lambda/I}\cmplx{\sheaf{F}}_I\wto
\cmplx{\sheaf{F}}_J.
$$
\end{enumerate}
Weak equivalences and cofibrations are defined as in Def.~\ref{defn:PDG(Lambda)}.
\end{defn}

For an arbitrary inverse system $\sheaf{F}\coloneqq(\sheaf{F}_I)_{I\in\openideals_{\Lambda}}$ of sheaves of $\Lambda$-modules on $X$ we let
 $\HF^k(X,\sheaf{F})$ denote the $k$-th continuous cohomology \cite[\S 3]{Jannsen:ContCohom}. Likewise, we write $\HF^k(\algc{X},\sheaf{F})$ for the continuous cohomology over the base change $\algc{X}$ of $X$ with respect to a fixed algebraic closure of $\FF$.

Assume that either $X$ is proper or that the characteristic of $\FF$ is a unit in $\Lambda$.
Using Godement resolution we construct in \cite[Def.~5.4.13]{Witte:PhD} Waldhausen exact functors
$$
\RDer\Sect(X,\cdot),\RDer\Sect(\algc{X},\cdot)\colon \cat{PDG}^{\cont}(X,\Lambda)\mto \cat{PDG}^{\cont}(\Lambda)
$$
such that for each $(\sheaf{F}_I)_{I\in\openideals_\Lambda}$ in $\cat{PDG}^{\cont}(X,\Lambda)$ concentrated in degree $0$, we have
\begin{align*}
\HF^i(X,(\sheaf{F}_I)_{I\in\openideals_\Lambda})&=\varprojlim_{I\in\openideals_{\Lambda}}\HF^i(X,\sheaf{F}_I)=\HF^i(\varprojlim_{I\in\openideals_{\Lambda}}\RDer\Sect(X,(\sheaf{F}_I)_{I\in\openideals_\Lambda})),\\
\HF^i(\algc{X},(\sheaf{F}_I)_{I\in\openideals_\Lambda})&=\varprojlim_{I\in\openideals_{\Lambda}}\HF^i(\algc{X},\sheaf{F}_I)=\HF^i(\varprojlim_{I\in\openideals_{\Lambda}}\RDer\Sect(\algc{X},(\sheaf{F}_I)_{I\in\openideals_\Lambda})).
\end{align*}
Moreover, there is an exact sequence
$$
0\mto\RDer\Sect(X,\cdot)\mto\RDer\Sect(\algc{X},\cdot)\xrightarrow{\id-\Frob_{\FF}} \RDer\Sect(\algc{X},\cdot)\mto 0
$$
where $\Frob_{\FF}$ denotes the geometric Frobenius over $\FF$ acting on $\algc{X}$ \cite[Prop.~6.1.2]{Witte:PhD}.

We also recall that given a second adic ring $\Lambda'$ and a $\Lambda'$-$\Lambda$-bimodule $M$ which is finitely generated and projective as $\Lambda'$-module, we may extend $\ringtransf_M$ to a Waldhausen exact functor
\begin{gather*}
\ringtransf_M\colon \cat{PDG}^{\cont}(X,\Lambda)\mto \cat{PDG}^{\cont}(X,\Lambda'),\\
(\cmplx{\sheaf{P}}_J)_{J\in\openideals_{\Lambda}}\mapsto (\varprojlim_{J\in\openideals_{\Lambda}} M/IM\tensor_{\Lambda}\cmplx{\sheaf{P}_J})_{I\in\openideals_{\Lambda'}}
\end{gather*}
such that
\begin{align*}
\ringtransf_M\RDer\Sect(X,\cmplx{\sheaf{P}})&\mto\RDer\Sect(X,\ringtransf_M(\cmplx{\sheaf{P}})),\\
\ringtransf_M\RDer\Sect(\algc{X},\cmplx{\sheaf{P}})&\mto\RDer\Sect(\algc{X},\ringtransf_M(\cmplx{\sheaf{P}}))
\end{align*}
are quasi-isomorphisms in $\cat{PDG}^{\cont}(\Lambda')$ \cite[Prop.~5.5.7]{Witte:PhD}.

\section{Selmer complexes}\label{sec:Selmer Complexes}

\noindent Selmer complexes have been introduced by Nekov{\'a}{\v{r}} in \cite{Nekovar:SelmerComplexes}. They are perfect complexes defined by modifying Galois cohomology by certain local conditions. In the geometric situation, one can identify the cohomology of the Selmer complexes with the \'etale cohomology of certain constructible sheaves. This is the point of view that we are adapting in the present article. In order to associate in a canonical manner classes in $\KTh$-groups to these complexes, we will construct them as elements of the Waldhausen categories introduced in the previous section.

We point out that the constructions given in this section also work in the number field case. The interested reader may consult \cite[Ch.~5]{Witte:Habil} for a detailed exposition.

\subsection{Notational conventions}\label{ss:conventions}

In the following, we let $K$ denote a fixed function field of characteristic $p>0$, i.\,e.\ a field of transcendence degree $1$ over $\FF_p$. We also fix a separable closure $\algc{K}$ of $K$ and write $\Gal_K$ for the Galois group of $\algc{K}/K$. Let $\FF$ denote the algebraic closure of $\FF_p$ inside $K$ and $\algc{\FF}$ the algebraic closure of $\FF$ inside $\algc{K}$. We let
$$
q\coloneqq p^{[\FF:\FF_p]}
$$
denote the number of elements of $\FF$.

For a fixed prime number $\ell$, we write $\FF_\cyc/\FF$ for the unique $\Int_\ell$-extension of $\FF$ and write $K_\cyc$ for the composite field $\FF_\cyc K$.

Write $C$ for the smooth and projective curve over $\FF$ whose closed points are the places of $K$. Note that by the definition of $\FF$, $C$ is geometrically connected. This is a convenient, but not necessary restriction. One can proceed as in \cite{Witte:MCVarFF} to deal also with the non-connected case.

For any open dense subscheme $U$ of $C$ and any extension field $L$ of $K$ inside $\algc{K}$, we write $U_L$ for the normalisation of $U$ in $L$ and   $j_L\colon \Spec L\mto U_L$ for the canonical inclusion of the generic point. More generally, for any open subscheme $V$ of $C$, we will write $j_{U}\colon U\mto V$ for the inclusion of an open subscheme $U$ of $V$ and $i_\Sigma\colon \Sigma\mto V$ for the inclusion of a closed subscheme $\Sigma$ of $V$.

Furthermore, we fix for each closed point $v\in C$ a strict henselisation of the local ring at $v$ inside $\algc{K}$. We write $\inertia_v\subset\decomp_v\subset\Gal_K$ for the corresponding inertia group and decomposition group, $\xi_v$ for the geometric point over $v$ corresponding to the residue field of the strict henselisation, and $\Frob_v$ for the corresponding geometric Frobenius at $\xi_v$. We let $\FF(v)$ denote the residue field of $v$ and $\deg(v)\coloneqq [\FF(v):\FF]$ the degree of $v$ over $\FF$.

For any compact or discrete abelian group $A$ we let
$$
\dual{A}\coloneqq\Hom_{\cont}(A,\Real/\Int)
$$
denote its Pontryagin dual.

If $\ell\neq p$, let $\cycchar\colon \Gal_K\mto \Int_{\ell}^{\times}$ denote the cyclotomic character:
$$
\sigma(\zeta)=\zeta^{\cycchar(\sigma)}
$$
for any $\ell^k$-th root of unity $\zeta\in\algc{K}$ and any $\sigma\in \Gal_K$. If $k\in\Int$ and $T$ is a $\Gal_K$-module, we let $T(k)$ denote the $k$-th Tate twist of $T$, i.\,e.\ the $\Gal_K$-module obtained from $T$ by multiplying the action of $\Gal_K$ by $\cycchar^k$.

\subsection{Admissible extensions and finitely ramified representations}\label{ss:AdmissibleLieExt}

Fix a prime number $\ell$.

\begin{defn}
An extension $K_\infty/K$ inside $\algc{K}$ is called an \emph{admissible extension} if
\begin{enumerate}
 \item $K_\infty/K$ is Galois and unramified outside a finite set of places.
 \item $K_\infty$ contains $K_\cyc$,
 \item $\Gal(K_\infty/K_\cyc)$ is a topologically finitely generated virtual pro-$\ell$-group.
\end{enumerate}
\end{defn}

If $K_\infty/K$ is an admissible extension, we let $G\coloneqq\Gal(K_\infty/K)$ denote its Galois group and set $H\coloneqq\Gal(K_\infty/K_\cyc)$, $\Gamma\coloneqq\Gal(K_\cyc/K)$. We may then choose a continuous splitting $\Gamma\mto G$ to identify $G$ with the corresponding semi-direct product $G=H\rtimes\Gamma$. Note that we do not assume that $G$ is an $\ell$-adic Lie group.

\begin{defn}
Let $\Lambda$ be an adic ring. We call a compact $\Lambda[[\Gal_K]]$-module $T$ a \emph{finitely ramified representation} of $\Gal_K$ over $\Lambda$ if
\begin{enumerate}
\item it is finitely generated and projective as $\Lambda$-module,
\item $T$ is unramified outside a finite set of places, i.\,e.\ the set
$$
\set{v\in C\given\text{$T^{\inertia_v}\neq T$}}
$$
is finite.
\end{enumerate}
\end{defn}

Recall that for a finite ring $R$, taking the stalk in the geometric point $\Spec \algc{K}$ is an equivalence of categories between the category of \'etale sheaves of $R$-modules on $\Spec K$ and the category of discrete $R[[\Gal_K]]$-modules \cite[Thm.~II.1.9]{Milne:EtCohom}. In our notation, we will not distinguish between the discrete $R[[\Gal_K]]$-module and the corresponding sheaf on $\Spec K$.

A finitely ramified representation $T$ of $\Gal_K$ over an adic ring $\Lambda$ then corresponds to a projective system of sheaves on $\Spec K$. We want to consider the system of direct image sheaves under the inclusion
\[
j_{F}\colon \Spec F\mto U.
\]
of the generic point into the open, dense subscheme $U$ of $C$. However, the naive definition, applying $j_{K*}$ to each element of the system, does not necessarily lead to an adic sheaf in our sense. We will consider a stabilised version instead, redefining the direct image sheaf as follows.

\begin{defn}\label{defn:modified direct image}
Let $\Lambda$ be an adic ring, $U\subset C$ an open non-empty subscheme and $T$ a finitely ramified representation of $\Gal_K$ over $\Lambda$. We define an inverse system of \'etale sheaves of $\Lambda$-modules $j_{K*}T\coloneqq(j_{K*}T_I)_{I\in\openideals_{\Lambda}}$ on $U$ by setting
$$
j_{K*}T_I\coloneqq\varprojlim_{J\in\openideals_{\Lambda}}\Lambda/I\tensor_{\Lambda}j_{K*}T/JT.
$$
\end{defn}

\begin{prop}\label{prop:stalk of modified direct image}
Let $\Lambda$ be an adic ring and $T$ be a finitely ramified representation of $\Gal_K$ over $\Lambda$ such that $T^{\inertia_v}$ is a finitely generated $\Lambda$-module for each closed point $v$ of $U$. Then $j_{K*}T_I$ is a constructible \'etale sheaf of $\Lambda/I$-modules on $U$ for any $I\in\openideals_{\Lambda}$. If $v$ is a closed point of $U$, then the stalk of $j_{K*}T_I$ in the geometric point $\xi_v$ is given by
$$
(j_{K*}T_I)_{\xi_v}
=\Lambda/I\tensor_{\Lambda}T^{\inertia_v}.
$$
In particular, $j_{K*}T$ is an object in $\cat{PDG}^{\cont}(U,\Lambda)$ if $T^{\inertia_v}$ is a finitely generated projective $\Lambda$-module for each closed point $v$ of $U$.
\end{prop}
\begin{proof}
Let $V$ be the open complement of $U$ by the set of points $v$ with $T^{\inertia_v}\neq T$. Consider a connected \'etale open set $W$ of $V$ and let $L\subset \algc{K}$ be the function field of $W$. Then for any $J\subset I$
$$
(\Lambda/I\tensor_{\Lambda}j_{K*}T/JT)(W)=(T/IT)^{\Gal_L}=(j_{K*}T/IT)(W).
$$
In particular, the restriction of $\Lambda/I\tensor_{\Lambda}j_{K*}T/JT$ to $V$ is a locally constant \'etale sheaf of $\Lambda/I$-modules which is independent of $J$. Now the category of \'etale sheaves of $\Lambda/I$-modules on $U$ which are locally constant on $V$ is equivalent to the category of tuples $(M,(M_v,\phi_v)_{v\in U-V})$ where $M$ is a discrete $\Lambda/I[[\Gal_K]]$-module unramified over $V$, the $M_v$ are discrete $\Gal_{\FF(v)}$-modules,  and $\phi_v\colon M_v\mto M^{\inertia_v}$ are homomorphisms of discrete $\Gal_{\FF(v)}$-modules \cite[Ex.~II.3.16]{Milne:EtCohom}.

By the above considerations, it is clear that the projective limit of the system $(\Lambda/I\tensor_{\Lambda}j_{K*}T/JT)_{J\in\openideals_{\Lambda}}$ exists in the category of \'etale sheaves of $\Lambda/I$-modules which are locally constant on $V$ and coincides with the projective limit taken in the category of all \'etale sheaves of $\Lambda/I$-modules. Moreover, it corresponds to the tuple
$$
(T/IT,(\varprojlim_{J\in\openideals_{\Lambda}}\Lambda/I\tensor_{\Lambda}(T/JT)^{\inertia_v},\phi_v\colon \varprojlim_{J\in\openideals_{\Lambda}}\Lambda/I\tensor_{\Lambda}(T/JT)^{\inertia_v}\mto (T/IT)^{\inertia_v})_{v\in U-V}).
$$
Beware that the projective limit
$$\varprojlim_{J\in\openideals_{\Lambda}}\Lambda/I\tensor_{\Lambda}(T/JT)^{\inertia_v}$$
is a priori taken in the category of discrete $\Lambda/I[[\Gal_{\FF(v)}]]$-modules (i.\,e.\ such that the stabiliser of every element is open in $\Gal_{\FF(v)}$).

In the category of abstract $\Lambda/I[[\Gal_{\FF(v)}]]$-modules, we have
$$
\varprojlim_{J\in\openideals_{\Lambda}}\Lambda/I\tensor_{\Lambda}(T/JT)^{\inertia_v}
=\Lambda/I\tensor_\Lambda\varprojlim_{J\in\openideals_{\Lambda}}(T/JT)^{\inertia_v}
=\Lambda/I\tensor_\Lambda T^{\inertia_v}.
$$
Here, the first equality is justified because projective limits of finite $\Lambda/I$-modules are exact and because $\Lambda/I$ is finitely presented as $\Lambda^{\op}$-module: In any adic ring $\Lambda$, the Jacobson radical $\Jac(\Lambda)$ is finitely generated both as left and as right ideal \cite[Thm.~36.39]{Warner:TopRings}. Therefore, the same is true for all open ideals $I\in\openideals_\Lambda$ and thus,  $\Lambda/I$ is a finitely presented $\Lambda^{\op}$-module.

By assumption, $T^{\inertia_v}$ is a finitely generated $\Lambda$-module. Hence,   $\Lambda/I\tensor_\Lambda T^{\inertia_v}$ is finite and the equality
$$
\varprojlim_{J\in\openideals_{\Lambda}}\Lambda/I\tensor_{\Lambda}(T/JT)^{\inertia_v}=\Lambda/I\tensor_\Lambda T^{\inertia_v}
$$
also holds in the category of discrete $\Lambda/I[[\Gal_{\FF(v)}]]$-modules. This shows that $j_{K*}T_I$ is constructible and that the stalks have the given form.

From the description of the stalks it is also immediate that
$$\Lambda/I\tensor_{\Lambda/J}j_{K*}T_J\isomorph j_{K*}T_I
$$
such that $j_{K*}T$ is indeed an object of $\cat{PDG}^{\cont}(U,\Lambda)$ if $T^{\inertia_v}$ is finitely generated and projective for all closed points $v$ in $U$.
\end{proof}

\begin{rem}\label{rem:non fg stalk}
Note that if $\Lambda$ is noetherian, $T^{\inertia_v}$ is automatically finitely generated. For general adic rings $\Lambda$, this is not true. Assume that $\ell$ is a prime dividing $p-1$ and let $\Lambda$ be the power series ring over $\FF_{\ell}$ in three non-commuting indeterminates $a,b,c$, modulo the relations $ab=0$, $ac=ca$, $(b+1)(c+1)=(c+1)(b+1)^p$. Set $K=\FF_p(t)$ and let $K_\infty=K_\cyc(\sqrt[\ell^\infty]{t})$ be the Kummer extension of $K_\cyc$ obtained by adjoining all $\ell^n$-th roots of $t$. Let $v$ be the point of $\Spec \FF_p[t]$ corresponding to the prime ideal $(t)$. Then $K_\infty/K$ is unramified over the complement $U$ of $v$ in $\Spec \FF_p[t]$ and $K_\infty/K_\cyc$ is totally and tamely
ramified in $v$. The Galois group $G=\Gal(K_\infty/K)$ is the pro-$\ell$-group topologically generated by two elements $\tau$ and $\sigma$, subject to the relation $\sigma\tau=\tau^p\sigma$. We obtain a finitely ramified representation $T$ of $\Gal_K$ over $\Lambda$ by letting $\tau$ act on $\Lambda$ by right multiplication with $b+1$ and $\sigma$ act by right multiplication with $c+1$. Hence, $T^{\inertia_v}$ is the kernel of the right multiplication with $b$, which is the left ideal of $\Lambda$ topologically generated by $ba^i$ for all $i>0$. Clearly, this ideal is not finitely generated.
\end{rem}

\begin{prop}\label{prop:limit property}
Let $\Lambda$ be an adic $\Int_\ell$-algebra, $T$ be a finitely ramified representation of $\Gal_K$ over $\Lambda$, and $U\subset V$ be open dense subschemes of $C$. Assume that $T^{\inertia_v}$ is a finitely generated $\Lambda$-module for all $v\in U$. If $\ell=p$, assume that $V=C$. Then
$$
\HF^k(V,j_{U!}j_{K*}T)=\varprojlim_{J\in\openideals_{\Lambda}}\HF^k(V,j_{U!}j_{K*}T/JT).
$$
for all $k\in\Int$.
\end{prop}
\begin{proof}
Let $(\sheaf{K}_J)_{J\in\openideals_{\Lambda}}$ and $(\sheaf{C}_J)_{J\in\openideals_{\Lambda}}$ denote the kernel and cokernel of the natural morphism of systems
$$
(j_{K*}T_J)_{J\in\openideals_{\Lambda}}\mto (j_{K*}T/JT)_{J\in\openideals_{\Lambda}}
$$
of \'etale sheaves on $U$. The restriction of $(\sheaf{K}_J)_{J\in\openideals_{\Lambda}}$ to the complement of
$$
\Sigma\coloneqq\set{v\in U\given\text{$T^{\inertia_v}\neq T$}}
$$
in $U$ is $0$. For $v\in\Sigma$ the stalk of $\sheaf{K}_J$ in the geometric point $\xi_v$ is a finite abelian group for each $J\in\openideals_{\Lambda}$. From Prop.~\ref{prop:stalk of modified direct image} we conclude
$$
T^{\inertia_v}=\varprojlim_{J\in\openideals_{\Lambda}}(j_{K*}T_J)_{\xi_v}=\varprojlim_{J\in\openideals_{\Lambda}}(j_{K*}T/JT)_{\xi_v}.
$$
Hence, the projective limit of the system $((\sheaf{K}_J)_{\xi_v})_{J\in\openideals_{\Lambda}}$ is $0$. It follows that the system must be Mittag-Leffler zero in the sense of \cite[Def.~1.10]{Jannsen:ContCohom}: For each natural number $n$ there exists a $m>n$ such that the transition maps
$$
(\sheaf{K}_{\Jac(\Lambda)^m})_{\xi_v}\mto (\sheaf{K}_{\Jac(\Lambda)^n})_{\xi_v}
$$
is the zero map. We conclude that the system of sheaves $(\sheaf{K}_J)_{J\in\openideals_{\Lambda}}$ are also Mittag-Leffler zero. The same remains true for $(j_{U!}\sheaf{K}_J)_{J\in\openideals_{\Lambda}}$ with $j_U\colon U\mto V$ denoting the natural inclusion. Now \cite[Lem.~1.11]{Jannsen:ContCohom} implies
$$
\HF^k(V,(j_{U!}\sheaf{K}_J)_{J\in\openideals_{\Lambda}})=0.
$$
The same argumentation also shows
$$
\HF^k(V,(j_{U!}\sheaf{C}_J)_{J\in\openideals_{\Lambda}})=0.
$$
Since the cohomology  groups $\HF^k(V,j_{U!}j_{K*}T/JT)$ are finite if $\ell\neq p$ \cite[Th. finitude, Cor.~1.10]{SGA4h} or if $V=C$  \cite[Cor.~VI.2.8]{Milne:EtCohom} we conclude
$$
\HF^k(V,j_{U!}j_{K*}T)=\varprojlim_{J\in\openideals_{\Lambda}}\HF^k(V,j_{U!}j_{K*}T/JT).
$$
\end{proof}

We will now fix an admissible extension $K_\infty/K$ with Galois group $G=H\rtimes \Gamma$. For a closed point $v$ of $C$ we will write $\kinert_v$ and $\ginert_v$ for the kernel and the image of the homomorphism $\inertia_v\mto G$, respectively. We also fix an open dense subscheme $U$ of $C$, an adic $\Int_\ell$-algebra $\Lambda$ and a finitely ramified representation $T$ of $\Gal_K$ over $\Lambda$. We let $\Lambda[[G]]^{\sharp}$ denote the $\Lambda[[G]][[\Gal_K]]$-module $\Lambda[[G]]$ with $g\in\Gal_K$ acting by the image of $g^{-1}$ in $G$ from the right. Note that $\Lambda[[G]]^{\sharp}\tensor_\Lambda T$ is a finitely ramified representation of $\Gal_K$ over $\Lambda[[G]]$.

\begin{cor}\label{cor:limit porperty for admissible extension}
Assume that $(\Lambda[[G]]^{\sharp}\tensor_\Lambda T)^{\inertia_v}$ is finitely generated for each closed point $v$ in $U$. If $\ell= p$, assume $V=C$. Then
$$
\HF^k(V,j_{U!}j_{K*}(\Lambda[[G]]^{\sharp}\tensor_\Lambda T))=\varprojlim_{K\subset_f L\subset K_\infty}\HF^k(V_L,j_{U_{L}!}j_{L*}T)
$$
for each $k\in\Int$, where $L$ runs through the finite Galois extensions of $K$ inside $K_\infty$.
\end{cor}
\begin{proof}
By Prop.~\ref{prop:limit property} we have
$$
\HF^k(V,j_{U!}j_{K*}(\Lambda[[G]]^{\sharp}\tensor_\Lambda T))=\varprojlim_{K\subset_f L\subset K_\infty}\HF^k(V,j_{U!}j_{K*}(\Lambda[\Gal(L/K)]^{\sharp}\tensor_\Lambda T)).
$$
Let $f\colon V_L\mto V$ denote the finite morphism of schemes corresponding to the finite extension $L/K$. Then
$$
\HF^k(V,j_{U!}j_{K*}(\Lambda[\Gal(L/K)]^{\sharp}\tensor_\Lambda T))=\HF^k(V,f_*j_{U_{L}!}j_{L*}T)=\HF^k(V_L,j_{U_{L}!}j_{L*}T)
$$
by \cite[Cor.~II.3.6]{Milne:EtCohom}.
\end{proof}

\begin{prop}\label{prop:constructability condition}
Assume that for every closed point $v$ of $U$ one of the following conditions is satisfied:
\begin{enumerate}
\item $T^{\kinert_v}=0$,
\item $\ginert_v$ contains an element of infinite order,
\item $T^{\kinert_v}$ is a finitely generated, projective $\Lambda$-module and $\ginert_v$ contains no element of order $\ell$.
\end{enumerate}
Then $j_{K*}(\Lambda[[G]]^{\sharp}\tensor_\Lambda T)$ is an object in $\cat{PDG}^{\cont}(U,\Lambda[[G]])$ and for every $I\in\openideals_{\Lambda[[G]]}$,
$$
(j_{K*}(\Lambda[[G]]^{\sharp}\tensor_\Lambda T)_I)_{\xi_v}=0
$$
if $v$ satisfies condition $(1)$ or $(2)$,
$$
(j_{K*}(\Lambda[[G]]^{\sharp}\tensor_\Lambda T)_I)_{\xi_v}=
\Lambda[[G]]/I\tensor_{\Lambda[[G]]}(\Lambda[[G]]^{\sharp}\tensor_\Lambda T^{\kinert_v})^{\ginert_v}
$$
if $v$ satisfies condition $(3)$.
\end{prop}
\begin{proof}
For each compact $\Lambda$-module $M$, write $\mathcal{U}_M$ for the lattice of open submodules of $M$. We note that $\Lambda[[G]]$ is a projective limit of finitely generated, free $\Lambda^{\op}$-modules and hence, a projective object in the category of compact $\Lambda^{\op}$-modules.  The completed tensor product
\[
\Lambda[[G]]\ctensor_{\Lambda} M=\varprojlim_{J\in\openideals_{\Lambda[[G]]}}\varprojlim_{U\in\mathcal{U}_M}\Lambda[[G]]/J\tensor_{\Lambda}M/U
\]
is thus an exact functor from the category of compact $\Lambda$-modules to the category of compact $\Lambda[[G]]$-modules. Moreover, we have
\[
\Lambda[[G]]\ctensor_{\Lambda}M=\Lambda[[G]]\tensor_{\Lambda}M
\]
if $M$ is finitely presented \cite[Prop. 1.14]{Witte:Splitting}.
In particular,
\[
\Lambda[[G]]\ctensor_{\Lambda} T^{\kinert_v}\isomorph (\Lambda[[G]]\ctensor_{\Lambda}T)^{\kinert_v}\isomorph (\Lambda[[G]]\tensor_{\Lambda}T)^{\kinert_v}.
\]

If $T^{\kinert_v}=0$, this obviously implies
\[
(\Lambda[[G]]\tensor_{\Lambda}T)^{\inertia_v}=0.
\]

Assume that $\ginert_v$ contains an element of infinite order and let $M$ be any finite $\Lambda$-module with a continuous $\ginert_v$-action. We can then find an element $\tau$ of infinite order in an $\ell$-Sylow subgroup of $\ginert_v$ which operates trivially on $M$. Consider the subgroup $\Upsilon\isomorph\Int_{\ell}$ of $\ginert_v$ which is topologically generated by $\tau$. By choosing a continuous map of profinite spaces
\[
G/\Upsilon \mto G
\]
that is a section of the projection map, we can view $\Lambda[[G]]^{\sharp}$ as a projective limit of finitely generated, free $\Lambda^{\op}[[\Upsilon]]$-modules and conclude that $1-\tau$ acts as non-zero divisor. In particular, we obtain an exact sequence of projective compact $\Lambda^{\op}$-modules
\[
0\mto\Lambda[[G]]^{\sharp}\xrightarrow{1-\tau}\Lambda[[G]]^{\sharp}\mto \Lambda[[G/\Upsilon]]\mto 0.
\]
The sequence remains exact after taking the tensor product over $\Lambda$ with $M$. Hence,
\[
(\Lambda[[G]]^{\sharp}\tensor_{\Lambda}M)^{\Upsilon}=
\ker\left(\Lambda[[G]]^{\sharp}\tensor_{\Lambda}M\xrightarrow{\id-\tau\tensor 1}\Lambda[[G]]^{\sharp}\tensor_{\Lambda}M\right)
=0.
\]

Recall that the powers of the Jacobson radical $\Jac(\Lambda)$ are finitely generated as left $\Lambda$-modules \cite[Thm.~36.39]{Warner:TopRings}. In particular, any finite $\Lambda$-module $M$ is finitely presented: For some $k$, the kernel $K$ of a surjection
\[
P\mto M
\]
with $P$ a finitely generated, free $\Lambda$-module contains the finitely generated module $\Jac(\Lambda)^kP$ as an open submodule. Hence, the tensor product of $M$ with a compact $\Lambda$-module agrees with the completed tensor product.

Writing the compact $\Lambda[[\ginert_v]]$-module $T^{\kinert_v}$ as projective limit of finite $\Lambda[[\ginert_v]]$-modules, we conclude
\[
(\Lambda[[G]]^{\sharp}\tensor_{\Lambda}T)^{\inertia_v}=(\Lambda[[G]]^{\sharp}\ctensor_{\Lambda}T^{\kinert_v})^{\ginert_v}=0.
\]

Assume that $\ginert_v$ contains no element of infinite order nor an element of order $\ell$ and that $T^{\kinert_v}$ is a finitely generated, projective $\Lambda$-module. Then $\ginert_v$ is a finite group of order $d$ prime to $\ell$. Set
\[
e_{\ginert_v}=\frac{1}{d}\sum_{\sigma\in\ginert_v}\sigma.
\]
Then $e_{\ginert_v}$ is a central idempotent in $\Lambda[\ginert_v]$ and
\[
(\Lambda[[G]]^{\sharp}\tensor_{\Lambda}T)^{\inertia_v}=e_{\ginert_v}(\Lambda[[G]]^{\sharp}\tensor_{\Lambda}T^{\kinert_v})
\]
is a finitely generated and projective $\Lambda[[G]]$-module.

We may now apply Prop.~\ref{prop:stalk of modified direct image} to conclude that $j_{K*}(\Lambda[[G]]^{\sharp}\tensor_\Lambda T)$ is an object in $\cat{PDG}^{\cont}(U,\Lambda[[G]])$.
\end{proof}

\begin{rem}\label{rem:principal covering}
If $K_{\infty}/K$ is unramified over $U$ and $f\colon U_{K_\infty}\mto U$ is the corresponding principal covering with Galois group $G$ in the sense of \cite[Def.~2.1]{Witte:MCVarFF}, then
$$
j_{K*}(\Lambda[[G]]^{\sharp}\tensor_\Lambda T)=f_!f^*j_{K*}T
$$
in the notation of \cite[\S 6]{Witte:MCVarFF}, see also \cite[Rem.~6.10]{Witte:MCVarFF}.
\end{rem}

\begin{defn}
Let $\Lambda$ be an adic $\Int_\ell$-algebra, $K_\infty/K$ an admissible extension, $v$ a closed point of $C$ and $T$ a finitely ramified representation of $\Gal_K$ over $\Lambda$.
\begin{enumerate}
\item We say that $T$ has projective stalks in $v$ if $T^{\inertia_v}$ is a finitely generated, projective $\Lambda$-module.
\item We say that $T$ has projective local cohomology in $v$ if $\ell\neq p$ and $\HF^1(\inertia_v,T)$ is a finitely generated, projective $\Lambda$-module.
\item We say that $T$ has projective stalks in $v$ over $K_\infty$ if $T^{\kinert_v}$ is finitely generated and projective.
\item We say that $T$ has projective local cohomology in $v$ over $K_\infty$ if $\ell\neq p$ and $\HF^1(\kinert_v,T)$ is a finitely generated, projective $\Lambda$-module.
\item We say that $T$ has ramification prime to $\ell$ in $v$ if the image of $\inertia_v$ in the automorphism group of $T$ has trivial $\ell$-Sylow subgroups.
\item We say that $K_\infty/K$ has ramification prime to $\ell$ in $v$ if $\ginert_v$ has trivial $\ell$-Sylow subgroups.
\item We say that $K_\infty/K$ has non-torsion ramification in $v$ if $\ginert_v$ contains an element of infinite order.
\end{enumerate}
\end{defn}

In particular, if $T$ has projective stalks over $U$, then $j_{K*}T$ is an object of $\cat{PDG}^{\cont}(U,\Lambda)$.

\begin{exmpl}\label{exmpl:ramification types}
Let $p$ denote the characteristic of $K$.
\begin{enumerate}
\item If $T$ has ramification prime to $\ell$ in a closed point $v$ and $\ell\neq p$, then it also has projective local cohomology in $v$. If $\ell=p$, then it has projective stalks in $v$.
\item If $T$ has projective local cohomology in $v$ (in $v$ over $K_\infty$), then it also has projective stalks in $v$ (in $v$ over $K_\infty$).
\item Assume $\ell\neq p$ and that $\Lambda$ has small finitistic projective dimension $0$, i.\,e.\ every finitely generated $\Lambda$-module of finite projective dimension is projective. Then $T$ has projective local cohomology in $v$ if and only if it has projective stalks in $v$. For example, this is true if $\Lambda$ is finite and commutative \cite[Thm. (Kaplansky)]{Bass:FinitisticDimension}. More generally, for any finite $\Lambda$, it is true precisely if the left annihilator of every proper right ideal of $\Lambda$ is non-zero \cite[Thm.~6.3]{Bass:FinitisticDimension}. It is not true for
    \[
    \Lambda=\begin{pmatrix}
    \Int/(\ell^2) & (\ell)/(\ell^2) \\
    \Int/(\ell^2) & \Int/(\ell^2)
    \end{pmatrix}
    \]
    \cite[Cor.~1.12]{KirkmanKuzSmall:Finitistic}.
\item If $\Lambda$ is noetherian of global dimension less or equal to $2$, then $T$ has projective stalks in all closed points $v$ of $C$, as $T^{\inertia_v}$ is the kernel of the continuous homomorphism of projective compact $\Lambda$-modules
   \[
   T\mto \prod_{\sigma\in\inertia_v}T,\qquad t\mapsto (t-\sigma t)_{\sigma\in\inertia_v}.
   \]
   As the global dimension is assumed to be less or equal to $2$, $T^{\inertia_v}$ is projective as compact $\Lambda$-module. As $\Lambda$ is noetherian, $T^{\inertia_v}$ is finitely generated and therefore, also projective as abstract $\Lambda$-module. The same argument shows that $T$ has projective stalks over $K_\infty$ in all closed points $v$ of $C$.
\comment{To do: Does projective stalks of $\Lambda[[G]]^\sharp$ in $v$ imply no elements of order $\ell$ in $\ginert_v$?}
\item Assume that $T$ has projective stalks  over $K_\infty$ and $K_\infty/K$ has ramification prime to $\ell$ in all closed points of $U$. Then $T$ has projective stalks over $U$. Moreover, $\Lambda[[G]]^\sharp\tensor_{\Lambda} T$ also has projective stalks over $U$ by Prop.~\ref{prop:constructability condition}, but neither of the two assumptions is necessary for this conclusion. The same remains true if one replaces ``projective stalks'' by ``projective local cohomology''.
\item It may happen that $T$ has projective stalks, but does not have projective stalks over $K_\infty$ in a closed point $v$ of $U$. For example, $T^{\inertia_v}$ can be trivial, while $T^{\kinert_v}$ is a non-trivial $\Lambda$-module that is not projective.
\item If $\ell\neq p$, then $K_\infty/K$ has non-torsion ramification in $v$ if and only if $\ginert_v$ is infinite. If $\ell=p$, it may happen that $\ginert_v$ is an infinite torsion group.
\end{enumerate}
\end{exmpl}

\begin{rem}
If $\Lambda$ is a noetherian adic $\Int_\ell$-algebra of finite global dimension, then one can modify Def.~\ref{defn:modified direct image} by choosing for each of the finitely many points $v$ for which $T^{\inertia_v}$ is not projective a resolution $\cmplx{P}$ of $T^{\inertia_v}$ by finitely generated, projective $\Lambda$-modules and replacing the stalk of $j_{K*}T_I$ in $v$ by $\Lambda/I\tensor_{\Lambda}\cmplx{P}$ for each open two-sided ideal $I$ of $\Lambda$.
\end{rem}



\subsection{Properties of Selmer complexes}

From now on, we fix two open dense subschemes $V$ and $W$ of $C$ with $V\cup W=C$ and set
$$
\Sigma_V\coloneqq C-V,\qquad \Sigma_W\coloneqq C-W,\qquad U\coloneqq V\cap W.
$$
We also fix a prime $\ell$, an adic $\Int_\ell$-algebra $\Lambda$ and a finitely ramified representation $T$ of $\Gal_K$ over $\Lambda$.

The complexes
$$
\cmplx{\SC}_{\Sigma_V,\Sigma_W}(T)\coloneqq\RDer\Sect(V,j_{U!}j_{K*}(\Lambda[[G]]^{\sharp}\tensor_\Lambda T))
$$
may be viewed as Selmer complexes in the sense of Nekov{\'a}{\v{r}} \cite[\S 6]{Nekovar:SelmerComplexes}, with unramified local conditions for each point $v$ of $U$ where $(\Lambda[[G]]^{\sharp}\tensor_\Lambda T)$ is ramified, full local conditions in each point of $\Sigma_V$ and empty local conditions in each point in $\Sigma_W$. Below, we will investigate their properties.

\begin{lem}\label{lem:perfectness}
Assume that $\Lambda[[G]]^{\sharp}\tensor_{\Lambda}T$ has projective stalks over $U$. If $\ell=p$, we also assume $V=C$. \begin{enumerate}
\item The complex
$$
\RDer\Sect(V,j_{U!}j_{K*}(\Lambda[[G]]^{\sharp}\tensor_\Lambda T))
$$
is an object of $\cat{PDG}^{\cont}(\Lambda[[G]])$.
\item If $U'\subset U$ is an open dense subscheme such that $K_\infty/K$ has non-torsion ramification over the complement of $U'$ in $U$, then
$$
\RDer\Sect(V,j_{U'!}j_{K*}(\Lambda[[G]]^{\sharp}\tensor_\Lambda T))\mto\RDer\Sect(V,j_{U!}j_{K*}(\Lambda[[G]]^{\sharp}\tensor_\Lambda T))
$$
is a weak equivalence in $\cat{PDG}^{\cont}(\Lambda[[G]])$.
\end{enumerate}
The same is true for $V$ replaced by $\algc{V}$.
\end{lem}
\begin{proof}
The first assertion follows from Prop.~\ref{prop:constructability condition} and the fact that the derived section functors over $V$ and $\algc{V}$ take objects of $\cat{PDG}^\cont(V,\Lambda[[G]])$ to objects of $\cat{PDG}^{\cont}(\Lambda[[G]])$.

We prove the second assertion. Let $j'\colon U'\mto U$ denote the inclusion map. By Prop.~\ref{prop:constructability condition} we have $j_{K*}(\Lambda[[G]]^\sharp\tensor_\Lambda T)_{\xi_v}=0$ for each $v\in U-U'$. Hence, the canonical morphism
$$
j'_!j_{K*}(\Lambda[[G]]^\sharp\tensor_\Lambda T)\mto j'_*j_{K*}(\Lambda[[G]]^\sharp\tensor_\Lambda T)
$$
is an isomorphism in $\cat{PDG}^{\cont}(U,\Lambda[[G]])$. The second assertion is an immediate consequence.
\end{proof}

\begin{rem}
In particular, by using Lem.~\ref{lem:perfectness}, we may exclude without loss of generality from $U$ all points in which $K_\infty/K$ has non-torsion ramification. We will also neglect the remaining points in which $K_\infty/K$ does not have ramification prime to $\ell$ or $T$ has no projective stalks over $K_\infty$, but $\Lambda[[G]]^{\sharp}\tensor_\Lambda T$ has projective stalks. These points may be considered as degenerate and it is not clear that their corresponding non-commutative Euler factors are well-behaved.
\end{rem}

Let $\Lambda'$ be another adic $\Int_{\ell}$-algebra and $\rho$ be a $\Lambda'$-$\Lambda[[G]]$-bimodule  which is finitely generated and projective as $\Lambda'$-module. We will write $T(\rho^\sharp)$ for the finitely ramified $\Gal_K$-representation $\rho\tensor_{\Lambda}T$ over $\Lambda'$, with $g\in\Gal_K$ acting by the image in $G$ of its inverse on $\rho$.

\begin{lem}\label{lem:commutativity with tensor product}
Assume that $K_\infty/K$ has ramification prime to $\ell$ over $U$. If $\ell=p$ we will also assume $V=C$.
\begin{enumerate}
\item If $T$ has ramification prime to $\ell$ over $U$, then $T(\rho^{\sharp})$ has ramification prime to $\ell$ over $U$ and
$$
\ringtransf_\rho\RDer\Sect(V,j_{U!}j_{K*}(\Lambda[[G]]^{\sharp}\tensor_\Lambda T))\mto \RDer\Sect(V,j_{U!}j_{K*}(T(\rho^\sharp)))
$$
is a weak equivalence in $\cat{PDG}^{\cont}(\Lambda')$.
\item If $T$ has projective local cohomology over $K_\infty$ in all closed points of $U$, then $T(\rho^{\sharp})$ has projective local cohomology over $U$ and
$$
\ringtransf_\rho\RDer\Sect(V,j_{U!}j_{K*}(\Lambda[[G]]^{\sharp}\tensor_\Lambda T))\mto \RDer\Sect(V,j_{U!}j_{K*}(T(\rho^\sharp)))
$$
is a weak equivalence in $\cat{PDG}^{\cont}(\Lambda')$.
\item If $T$ has projective stalks over $K_\infty$ in all closed points of $U$ and $\rho$ is projective as compact $\Lambda^{\op}$-module, then $T(\rho^\sharp)$ has projective stalks over $U$ and the canonical morphism
$$
\ringtransf_\rho\RDer\Sect(V,j_{U!}j_{K*}(\Lambda[[G]]^{\sharp}\tensor_\Lambda T))\mto \RDer\Sect(V,j_{U!}j_{K*}(T(\rho^\sharp)))
$$
is a weak equivalence in $\cat{PDG}^{\cont}(\Lambda')$.
\end{enumerate}
The same is true for $V$ replaced by $\algc{V}$.
\end{lem}
\begin{proof}
As explained in Section \ref{ss:perfect complexes of adic sheaves}, $\ringtransf_\rho$ commutes with $\RDer\Sect(V,\cdot)$ and $\RDer\Sect(\algc{V},\cdot)$.
So, we need to prove that
$$
\ringtransf_\rho j_{K*}(\Lambda[[G]]^{\sharp}\tensor_\Lambda T)\mto j_{K*}(T(\rho^\sharp))
$$
is an isomorphism in $\cat{PDG}^{\cont}(U,\Lambda')$. Since this can be checked on stalks, we need to prove that
$$
\rho\tensor_{\Lambda[[G]]}(\Lambda[[G]]^{\sharp}\tensor_\Lambda T)^{\inertia_v}=(T(\rho^\sharp))^{\inertia_v}
$$
for all closed points $v\in U$. Since $K_\infty/K$ has ramification prime to $\ell$ in $v$, the $\ell$-Sylow subgroup of $\ginert_v$ is trivial such that taking invariants under $\ginert_v$ is an exact functor on the category of compact $\Int_{\ell}[[\ginert_v]]$-modules. Moreover, $T^{\kinert_v}$ is finitely generated and projective as $\Lambda$-module by assumption. Hence
$$
\rho\tensor_{\Lambda[[G]]}(\Lambda[[G]]^{\sharp}\tensor_\Lambda T^{\kinert_v})^{\ginert_v}=(\rho^{\sharp}\tensor_\Lambda T^{\kinert_v})^{\ginert_v}.
$$

If $\rho$ is projective as compact $\Lambda^{\op}$-module, then taking the completed tensor product with $\rho^\sharp$ over $\Lambda$ is an exact functor on compact $\Lambda$-modules. Moreover, the completed tensor product commutes with arbitrary direct products and agrees with the usual tensor product on finitely presented modules \cite[Prop.~1.7, Prop.~1.14]{Witte:Splitting}. By taking the completed tensor product with $\rho^{\sharp}$ of the left exact sequence
\[
0\mto T^{\kinert_v}\mto T\xrightarrow{x\mapsto (x-\sigma x)_{\sigma \in\kinert_v}}\prod_{\sigma\in\kinert_v} T,
\]
we obtain
\[
\rho^\sharp\tensor_{\Lambda} T^{\kinert_v}=\rho^\sharp\ctensor_{\Lambda} T^{\kinert_v}=(\rho^\sharp\tensor_{\Lambda} T)^{\kinert_v},
\]
as desired.

If $\ell$ is different from the characteristic of $F$, the $\Tor$ spectral sequence for the derived tensor product of $\rho^\sharp$ with the cochain complex of the $\kinert_v$-module $T$ gives us an exact sequence
\[
0\mto\Tor_2^{\Lambda}(\rho^\sharp,\HF^1(\kinert_v,T))\mto \rho^\sharp\tensor_{\Lambda} T^{\kinert_v}\mto \Tor_1^{\Lambda}(\rho^\sharp,\HF^1(\kinert_v,T))\mto 0
\]
Hence, if $\HF^1(\kinert_v,T)$ is finitely generated and projective as a $\Lambda$-module, then
$$
\rho^\sharp\tensor_\Lambda T^{\kinert_v}\isomorph(\rho^\sharp\tensor_\Lambda T)^{\kinert_v}.
$$

If $T$ has ramification prime to $\ell$, then one can replace $\kinert_v$ by its image in the automorphism group of $T$. Since this group is of order prime to $\ell$, the natural map
$$
\rho^\sharp\tensor_\Lambda T^{\kinert_v}\mto(\rho^\sharp\tensor_\Lambda T)^{\kinert_v}
$$
is again an isomorphism. This completes the proof of the lemma.
\end{proof}

\begin{lem}\label{lem:S torsion of additional Euler factors}
Assume that $K_\infty/K$ has ramification prime to $\ell$ and that $T$ has projective stalks over $K_\infty$ in the closed point $v$ of $U$. Write $i_v\colon v\mto U$ for the closed immersion. Then
\[
\RDer\Sect(v,i_v^*j_{K*}(\Lambda[[G]]^\sharp\tensor_{\Lambda}T))
\]
is in $\cat{PDG}^{\cont,w_H}(\Lambda[[G]])$.
\end{lem}
\begin{proof}
The complex $\RDer\Sect(v,i_v^*j_{K*}(\Lambda[[G]]^\sharp\tensor_{\Lambda}T))$ may be identified with the strictly perfect complex of $\Lambda[[G]]$-modules
\[
\cmplx{C}\colon\quad (\Lambda[[G]]^\sharp\tensor_{\Lambda}T)^{\inertia_v}\xrightarrow{\id-\Frob_v}(\Lambda[[G]]^\sharp\tensor_{\Lambda}T)^{\inertia_v}
\]
sitting in degree $0$ and $1$. Choose an open pro-$\ell$-subgroup $H'$ of $H$ which is normal in $G$. By \cite[Prop.~4.8]{Witte:MCVarFF}, it suffices to show that $\Lambda/\Jac(\Lambda)[[G/H']]\tensor_{\Lambda[[G]]}\cmplx{C}$ has finite cohomology groups. Let $Z$ be the centre of $\Lambda/\Jac(\Lambda)$, which is a finite product of finite fields of characteristic $\ell$ and consider
\[
\begin{split}
P&\coloneqq\Lambda/\Jac(\Lambda)[[G/H']]\tensor_{\Lambda[[G]]}(\Lambda[[G]]^\sharp\tensor_{\Lambda}T^{\inertia_v})\\
&\isomorph (\Lambda/\Jac(\Lambda)[[G/H']]^\sharp\tensor_{\Lambda/\Jac(\Lambda)}T^{\kinert_v}/\Jac(\Lambda)T^{\kinert_v})^{\ginert_v}
\end{split}
\]
as finitely generated, projective $Z[[\Gamma]]$-module. Choose $n$ large enough, such that $\Frob_v^{n}$ operates trivially on the finite groups $\ginert_v$ and $T^{\kinert_v}/\Jac(\Lambda)T^{\kinert_v}$. Then
\[
\id-\Frob_v^n=(\id-\Frob_v)(\sum_{s=0}^{n-1}\Frob_v^s)
\]
is an injective endomorphism of $P$. The same is then also true for $\id-\Frob_v$. We conclude from the elementary divisor theorem that the cokernel of $\id-\Frob_v$ is finite, as desired.
\end{proof}

\begin{lem}\label{lem:HS argument}
Assume that $H=\Gal(K_\infty/K_\cyc)$ and $\Lambda$ are finite. If $\ell=p$, we also assume that $V=C$. Then for each $k\in\Int$
\begin{align*}
\HF^{k}(V,j_{U!}j_{K*}(\Lambda[[G]]^\sharp\tensor_\Lambda T))&=\HF^{k-1}(V_{K_{\infty}},j_{U_{K_{\infty}}!}j_{K_{\infty}*}T)\\
&=\HF^{k-1}(V_{\algc{\FF}K_{\infty}},j_{U_{\algc{\FF}K_{\infty}}!}j_{\algc{\FF}K_{\infty}*}T)^{\Gal(\algc{\FF}K_{\infty}/K_{\infty})}.
\end{align*}
In particular, $\HF^k(V,j_{U!}j_{K*}(\Lambda[[G]]^\sharp\tensor_\Lambda T))$ is finitely generated as $\Lambda$-module.
\end{lem}
\begin{proof}
Note that $\Lambda[[G]]$ is noetherian. Hence, $(\Lambda[[G]]^\sharp\tensor_{\Lambda}T)^{\inertia_v}$ is finitely generated as a $\Lambda[[G]]$-module for each closed point $v$ of $U$. We can apply Cor.~\ref{cor:limit porperty for admissible extension} and then proceed as in the proof of \cite[Prop.~13]{Witte:Survey}.
\end{proof}

\section{Noncommutative main conjectures}\label{sec:nc main conjectures}

\noindent As before, we let $\Lambda$ denote an adic $\Int_\ell$-algebra and fix an admissible extension $K_\infty/K$ with Galois group $G=H\rtimes\Gamma$, a finitely ramified representation $T$ of $\Gal_K$ over $\Lambda$, and two open dense subschemes $V$ and $W$ of the proper smooth curve $C$ with function field $K$ such that
$$
C=V\cup W,\qquad U=V\cap W,\qquad \Sigma_V=C-V,\qquad \Sigma_W=C-W.
$$
If $X$ is any open subscheme of $C$, we will write $X^0$ for its set of closed points. We will also continue to use the general conventions from the beginning of Section~\ref{sec:Selmer Complexes}.

\subsection{\texorpdfstring{$L$}{L}-functions}

Assume that $\Lambda$ is a commutative adic $\Int_\ell$-algebra. Below, we will define a $\Sigma_W$-truncated $\Sigma_V$-modified $L$-function of $T$. We need to assume that $T$ has projective stalks over $U$. Usually, one also requires that $\Sigma_V$ does not contain any $v$ for which $T$ is ramified. If $\ell\neq p$, however, there exists a natural extension of the usual definition. In the case $\ell=p$ one can still give a sensible definition under the condition that $T$ is at most tamely ramified for all $v\in \Sigma_V$.

Under our ramification assumption we know by Prop.~\ref{prop:stalk of modified direct image} that $j_{K*}T$ is an object in $\cat{PDG}^{\cont}(U,\Lambda)$. Hence, assuming $\ell\neq p$, the complex $i_v^*\RDer j_{V*}j_{U!}j_{K*}T$ is in $\cat{PDG}^{\cont}(v,\Lambda)$ and we may define its $L$-function
$$
L(i_v^*\RDer j_{V*}j_{U!}j_{K*}T,t)\in \KTh_1(\Lambda[[t]])\isomorph\Lambda[[t]]^\times.
$$
as in \cite[\S 8]{Witte:MCVarFF}.

In more concrete terms, we can proceed as follows. Let $\decomp_{v,\ell}$ denote the quotient of $\decomp_v$ by the group $\inertia_{v,\neg\ell}$ generated by all $\ell'$-Sylow subgroups of $\inertia_v$ for $\ell'\neq\ell$.  Fix a choice of $\tau,\varphi\in \decomp_{v,\ell}$ such that $\tau$ topologically generates the image of the inertia group in $\decomp_{v,\ell}$ and $\varphi$ is a lift of the geometric Frobenius $\Frob_v$. It is well known that $\tau$ and $\varphi$ topologically generate $\decomp_{v,\ell}$ and that
$$
\varphi\tau\varphi^{-1}=\tau^{-q^{\deg(x)}}
$$
\cite[Thm.~7.5.2]{NSW:CohomNumFields}. We define a complex $\cmplx{D}_v(T)$ of $\Lambda[[\Gal_{\FF(v)}]]$-modules as follows. For $k\neq 0,1$ we set $D^k_v(T)\coloneqq 0$. As $\Lambda$-modules we have $D^0_v(T)\coloneqq D^1_v(T)\coloneqq T^{\inertia_{v,\neg\ell}}$ and the differential is given by $\id-\tau$. The geometric Frobenius $\Frob_v$ acts on $D^0_v(T)$ via $\varphi$ and on $D^1_v(T)$ via
$$
\varphi\left(\sum_{m=0}^{q^{\deg(v)}-1}\tau^m\right)\in \Lambda[[\decomp_{v,\ell}]]^\times.
$$
The complex $\cmplx{D}_v(T)$ may be viewed as an object in $\cat{PDG}^{\cont}(v,\Lambda)$ which is weakly equivalent to $i_v^*(\RDer j_{V*}j_{U!}j_{K*}T)$ for $v\in\Sigma_V$ and we have
$$
L(i_v^*\RDer j_{V*}j_{U!}j_{K*}T,t)=\det(1-t^{\deg(v)}\Frob_v\lcirclearrowright D^0_v(T))^{-1}
\det(1-t^{\deg(v)}
\Frob_v\lcirclearrowright D^1_v(T)).
$$
Note that if $T$ has ramification prime to $\ell$ in $v$, then $D^0_v(T)=D^1_v(T)=T^{\inertia_v}$ and
$$
L(i_v^*\RDer j_{V*}j_{U!}j_{K*}T,t)=\det(1-t^{\deg(v)}\Frob_v\lcirclearrowright T^{\inertia_v})^{-1}\det(1-(tq)^{\deg(v)}\Frob_v\lcirclearrowright T^{\inertia_v}).
$$
In the case that $\ell\neq p$, we can take this line as a definition for $L(i_v^*\RDer j_{V*}j_{U!}j_{K*}T,t)$ if $T$ is at most tamely ramified in $v$.

\begin{defn}\label{defn:L-function product formula}
Assume that $T$ has projective stalks over $U$. If $\ell=p$ we must also assume that $T$ is at most tamely ramified at $v$ for all $v\in \Sigma_V$. The \emph{$\Sigma_W$-truncated $\Sigma_V$-modified $L$-function of $T$} is given by
$$
L_{\Sigma_W,\Sigma_V}(T,t)\coloneqq\prod_{v\in U^0}\det(1-\Frob_v t^{\deg(v)}\lcirclearrowright T^{\inertia_v})^{-1}\prod_{v\in \Sigma_V}L(i_{v}^*\RDer j_{V*}j_{U!}j_{K*}T,t)
$$
\end{defn}

Set
$$
\Lambda\langle t \rangle\coloneqq\varprojlim_{I\in\openideals_\Lambda}\Lambda/I[t]
$$
and consider the multiplicatively closed subset
$$
\tilde{S}\coloneqq\set{f(t)\in \Lambda \langle t \rangle\given\text{ $f(0)\in \Lambda^{\times}$}}\subset \Lambda\langle t \rangle.
$$
Let $\gamma$ be the image of $\Frob_{\FF}$ in $\Gamma=\Gal(\FF_\cyc/\FF)$. We recall from \cite[Lem.~1.10]{Witte:Survey} that
$$
\Lambda\langle t \rangle_{\tilde{S}}\mto \Lambda[[\Gamma]]_S,\qquad t\mapsto \gamma^{-1}
$$
is a ring homomorphism.

It follows from \cite[Thm. 8.6]{Witte:MCVarFF} ($\ell\neq p$) and \cite[Thm. 5.5]{Witte:UnitLFunctions} ($\ell=p$) that
$$
L_{\Sigma_W,\Sigma_V}(T,t)\in \Lambda\langle t \rangle_{\tilde{S}}^\times
$$
such that we obtain elements
$$
L_{\Sigma_W,\Sigma_V}(T,\gamma^{-1})\in \Lambda[[\Gamma]]_S^\times.
$$
These elements are the natural analogues of the classical $\ell$-adic $L$-function in the number field case.

\begin{rem}
It is possible to extend the construction of $L_{\Sigma_W,\Sigma_V}(T,t)$ also to the case that $\Lambda$ is a non-commutative adic $\Int_\ell$-algebra. If $\Lambda$ is noetherian, then the subset $\tilde{S}\subset\Lambda\langle t\rangle$ is still a left denominator set and one can define $L_{\Sigma_W,\Sigma_V}(T,t)$ as an element of $\KTh_1(\Lambda\langle t\rangle_{\tilde{S}})$. If $\Lambda$ is not noetherian, then one has to replace $\KTh_1(\Lambda\langle t\rangle_{\tilde{S}})$ by the first $\KTh$-group of a certain Waldhausen category. The image of $L_{\Sigma_W,\Sigma_V}(T,t)$ in $\KTh_1(\Lambda[[t]])$ satisfies the product formula from Def.~\ref{defn:L-function product formula}. Note, however, that $\KTh_1(\Lambda\langle t\rangle_{\tilde{S}})\mto\KTh_1(\Lambda[[t]])$ does not need to be injective.
\end{rem}

\begin{rem}
If $T=\rho$ is an Artin representation of $\Gal_K$ over $\Int_\ell$, note that our definition of $L_{\Sigma_W,\Sigma_V}(\rho,t)$ follows Grothendieck's convention. If one uses Artin's original definition, then it corresponds to the $L$-function of the dual representations of $\rho$, see \cite[Ex.~V.2.20]{Milne:EtCohom}.
\end{rem}

\subsection{The main conjecture for Galois representations}\label{ss:MCforladicReps}

In this section we will discuss non-commutative Iwasawa Main Conjectures for the finitely ramified representation  $T$ of $\Gal_K$ over any adic $\Int_\ell$-algebra $\Lambda$. We will need the Waldhausen categories $\cat{PDG}^{\cont,w_H}(\Lambda[[G]])$ and $w_H\cat{PDG}^{\cont}(\Lambda[[G]])$ from Def.~\ref{defn:wHPDG}. Recall also the evaluation map $\eval_\rho$ from \eqref{eq:def of evaluation map}.

First, we treat the case $\ell\neq p$.

\begin{thm}\label{thm:MCforRep l p different}
Assume that $\ell\neq p$. Further, assume that $K_\infty/K$ has ramification prime to $\ell$ and that $T$ has projective stalks over $K_\infty$ in all closed points of $U$. Then
\begin{enumerate}
\item $\RDer\Sect(V,j_{U!}j_{K*}(\Lambda[[G]]^{\sharp}\tensor_\Lambda T))$ is in $\cat{PDG}^{\cont,w_H}(\Lambda[[G]])$ and the endomorphism $\id-\Frob_{\FF}$ of $\RDer\Sect(\algc{V},j_{U!}j_{K*}(\Lambda[[G]]^{\sharp}\tensor_\Lambda T))$ is a weak equivalence in $w_H\cat{PDG}^{\cont}(\Lambda[[G]])$.
\item Set
$$
\ncL_{K_{\infty}/K,\Sigma_W,\Sigma_V}(T)\coloneqq[\id-\Frob_{\FF}\lcirclearrowright \RDer\Sect(\algc{V},j_{U!}j_{K*}(\Lambda[[G]]^{\sharp}\tensor_\Lambda T))]^{-1}
$$
in $\KTh_1(\Lambda[[G]]_S)$. Then
$$
\bh \ncL_{K_{\infty}/K,\Sigma_W,\Sigma_V}(T)=-[\RDer\Sect(V,j_{U!}j_{K*}(\Lambda[[G]]^{\sharp}\tensor_\Lambda T))]
$$
\item Let $\Lambda'$ be a commutative adic $\Int_\ell$-algebra and  $\rho$ be a $\Lambda'$-$\Lambda[[G]]$-bimodule which is finitely generated and projective as $\Lambda'$-module. Assume either that $T$ has projective local cohomology over $K_\infty$ in all closed points of $U$ or that $\rho$ is projective as compact $\Lambda^{\op}$-module. Then
    $$
    \eval_\rho(\ncL_{K_{\infty}/K,\Sigma_W,\Sigma_V}(T))=L_{\Sigma_W,\Sigma_V}(T(\rho^\sharp),\gamma^{-1}).
    $$
\end{enumerate}
\end{thm}
\begin{proof}
Set
$$
\sheaf{F}\coloneqq j_{U!}j_{K*}(\Lambda[[G]]^{\sharp}\tensor_\Lambda T).
$$
We begin by showing $(1)$. We first assume that $T$ has only ramification prime to $\ell$ over $U$ and that $\Lambda$ and  $H$ are finite. By Lem.~\ref{lem:perfectness} the complex $\RDer\Sect(V,\sheaf{F})$
is an object in $\cat{PDG}^{\cont}(\Lambda[[G]])$. Hence, we can find a strictly perfect complex $\cmplx{D}$ of $\Lambda[[G]]$-modules and a quasi-isomorphism
$$
\cmplx{D}\mto \varprojlim_{I\in\openideals_{\Lambda[[G]]}}\RDer\Sect(V,\sheaf{F}_I).
$$
By Lem.~\ref{lem:HS argument} the cohomology groups of $\cmplx{D}$ are finite $\Lambda[H]$-modules. Since $\Lambda[H]$ is noetherian, every finitely generated projective $\Lambda[[G]]$-module is flat as $\Lambda[H]$-module. Hence, the complex $\cmplx{D}$ is also of finite flat dimension over $\Lambda[H]$. Hence, it is a perfect complex of $\Lambda[H]$-modules and $\RDer\Sect(V,\sheaf{F})$ is in $\cat{PDG}^{\cont,w_H}(\Lambda[[G]])$.

For general $\Lambda$ and  $H$ we may choose an open pro-$\ell$-subgroup $H'$ in $H$ which is normal in $G$. By \cite[Prop.~4.8]{Witte:MCVarFF} an object in $\cat{PDG}^{\cont}(\Lambda[[G]])$ is an object in $\cat{PDG}^{\cont,w_H}(\Lambda[[G]])$ if its image under $\ringtransf_{\Lambda/\Jac(\Lambda)[[G/H']]}$ is an object in $\cat{PDG}^{\cont,w_{H/H'}}(\Lambda/\Jac(\Lambda)[[G/H']])$. But
\[
\begin{aligned}
\ringtransf_{\Lambda/\Jac(\Lambda)[[G/H']]}(\RDer\Sect(&V,j_{U!}j_{K*}(\Lambda[[G]]^{\sharp}\tensor_\Lambda T)))\\
&\downarrow\\
\RDer\Sect(&V,j_{U!}j_{K*}(\Lambda/\Jac(\Lambda)[[G/H']]^{\sharp}\tensor_\Lambda T))
\end{aligned}
\]
is a weak equivalence by Lem.~\ref{lem:commutativity with tensor product}.

For general $U$, let $i_Z\colon Z\mto U$ be the closed reduced subscheme over which $T$ does not have ramification prime to $\ell$ and write $j'\colon U'\mto U$ for the open complement of $Z$ in $U$. Consider the generic point $j_K\colon \Spec K\mto U'$. Then
\begin{multline*}
0\mto \RDer\Sect(V,j_{U!}j'_!j_{K*}(\Lambda[[G]]^{\sharp}\tensor_\Lambda T))\mto \RDer\Sect(V,j_{U!}j'_*j_{K*}(\Lambda[[G]]^{\sharp}\tensor_\Lambda T))\\
\mto \RDer\Sect(Z,i_Z^*j'_*j_{K*}(\Lambda[[G]]^{\sharp}\tensor_\Lambda T))\mto 0
\end{multline*}
is an exact sequence in $\cat{PDG}^{\cont}(\Lambda[[G]])$. The two outer terms are objects in $\cat{PDG}^{\cont,w_H}(\Lambda[[G]])$ by our previous considerations and by Lem.~\ref{lem:S torsion of additional Euler factors}. Hence, so is the middle term.

Finally, recall from Section~\ref{ss:perfect complexes of adic sheaves}
that we have an exact sequence
$$
0\mto \RDer\Sect(V,\sheaf{F})\mto \RDer\Sect(\algc{V},\sheaf{F})\xrightarrow{\id-\Frob_{\FF}}\RDer\Sect(\algc{V},\sheaf{F})\mto 0
$$
in $\cat{PDG}^{\cont}(\Lambda[[G]])$. In particular, the endomorphism $\id-\Frob_{\FF}$ is a weak equivalence in $w_H \cat{PDG}^{\cont}(\Lambda[[G]])$. This finishes the proof of $(1)$.

By the explicit description of the boundary homomorphism $\bh$ from Section~\ref{ss:Framework}
we conclude
$$
\bh [\id-\Frob_{\FF}]=[\RDer\Sect(V,j_{U!}j_{K*}(\Lambda[[G]]^{\sharp}\tensor_\Lambda T))].
$$
This proves $(2)$.

For $(3)$ we use Lem.~\ref{lem:commutativity with tensor product}
to show
$$
\eval_\rho(\ncL_{K_{\infty}/K,\Sigma_W,\Sigma_V}(T))=\ncL_{K_\cyc/K,\Sigma_W,\Sigma_V}(T(\rho^\sharp)).
$$
Let $(f\colon C_{K_\cyc}\mto C,\Gamma)$ denote the principal covering with Galois group $\Gamma$ in the sense of \cite[Def.~2.1]{Witte:MCVarFF} and consider the Waldhausen exact functor
$$
f_!f^*\colon \cat{PDG}^{\cont}(C,\Lambda)\mto\cat{PDG}^{\cont}(C,\Lambda[[\Gamma]])
$$
from \cite[Def.~6.1]{Witte:MCVarFF}. Since $f$ is everywhere unramified, one checks easily on the stalks that the natural morphism
$$
f_!f^*(\RDer j_{V*})j_{U!}j_{K*}(T)\mto (\RDer j_{V*})j_{U!}j_{K*}(\Lambda[[\Gamma]]^{\sharp}\tensor_\Lambda T)
$$
is a weak equivalence in $\cat{PDG}^{\cont}(C,\Lambda[[\Gamma]])$. From \cite[Thm.~8.6]{Witte:MCVarFF} we thus conclude
$$
\ncL_{K_\cyc/K,\Sigma_W,\Sigma_V}(T)=L((\RDer j_{V*})j_{U!}j_{K*}T,\gamma^{-1})=L_{\Sigma_W,\Sigma_V}(T,\gamma^{-1})
$$
and hence,
$$
\eval_\rho(\ncL_{K_{\infty}/K,\Sigma_W,\Sigma_V}(T))=L_{\Sigma_W,\Sigma_V}(T(\rho^\sharp),\gamma^{-1})
$$
as claimed.
\end{proof}

Assume that $\ell=p$. If $\Lambda[[G]]^{\sharp}\tensor_\Lambda T$ has projective stalks over $U$, then we can associate to $j_{U!}j_{K*}(\Lambda[[G]]^{\sharp}\tensor_\Lambda T)$ an element
$$
Q(j_{U!}j_{K*}(\Lambda[[G]]^{\sharp}\tensor_\Lambda T),t)\in\KTh_1(\Lambda[[G]]\langle t \rangle)
$$
that measures the failure of the Grothendieck trace formula \cite[Thm. 4.1]{Witte:UnitLFunctions}. It is well-determined up to elements in the kernel of the surjection
\[
\KTh_1(\Lambda[[G]]\langle t \rangle)\mto \varprojlim_{I\in\openideals_{\Lambda[[G]]}}\KTh_1(\Lambda[[G]]/I[t]).
\]
We will write
$$
q(\Lambda[[G]]^{\sharp}\tensor_\Lambda T)\coloneqq Q(j_{U!}j_{K*}(\Lambda[[G]]^{\sharp}\tensor_\Lambda T),1)\in\KTh_1(\Lambda[[G]]).
$$
By \cite[Prop. 4.3]{Witte:UnitLFunctions}, this element is in fact independent of $U$.

\begin{thm}\label{thm:MCforRep l equal p myLfunc}
Assume that $\ell = p$, that $K_\infty/K$ has ramification prime to $p$ and $T$ has projective stalks over $K_\infty$ in each closed point of $U$, and that $K_\infty/K$ and $T$ have ramification prime to $p$ in each point of $\Sigma_V$. Let $j_K\colon \Spec K\mto W$ be the generic point of $W$. Then
\begin{enumerate}
\item $\RDer\Sect(C,j_{W!}j_{K*}(\Lambda[[G]]^{\sharp}\tensor_\Lambda T))$ is in $\cat{PDG}^{\cont,w_H}(\Lambda[[G]])$ and the endomorphism $\id-\Frob_{\FF}$ of $\RDer\Sect(\algc{C},j_{W!}j_{K*}(\Lambda[[G]]^{\sharp}\tensor_\Lambda T))$ is a weak equivalence in $w_H\cat{PDG}^{\cont}(\Lambda[[G]])$.
\item Set
\begin{multline*}
\ncL_{K_{\infty}/K,\Sigma_W,\Sigma_V}(T)\coloneqq[\id-\Frob_{\FF}\lcirclearrowright \RDer\Sect(\algc{C},j_{W!}j_{K*}(\Lambda[[G]]^{\sharp}\tensor_\Lambda T))]^{-1}\\
q(\Lambda[[G]]^{\sharp}\tensor_\Lambda T)\prod_{v\in\Sigma_V}[\id-\Frob_v q^{\deg(v)}\lcirclearrowright (\Lambda[[G]]^\sharp \tensor_\Lambda T)^{\inertia_v}]
\end{multline*}
in $\KTh_1(\Lambda[[G]]_S)$. Then
$$
\bh \ncL_{K_{\infty}/K,\Sigma_W,\Sigma_V}(T)=-[\RDer\Sect(C,j_{W!}j_{K*}(\Lambda[[G]]^{\sharp}\tensor_\Lambda T))]
$$
\item Let $\Lambda'$ be a commutative adic $\Int_\ell$-algebra and  $\rho$ be a $\Lambda'$-$\Lambda[[G]]$-bimodule which is finitely generated and projective as $\Lambda'$-module. Assume either that $T$ has only ramification prime to $p$ over $U$ or that $\rho$ is projective as compact $\Lambda^{\op}$-module. Then
    $$
    \eval_\rho(\ncL_{K_{\infty}/K,\Sigma_W,\Sigma_V}(T))=L_{\Sigma_W,\Sigma_V}(T(\rho^\sharp),\gamma^{-1}).
    $$
\end{enumerate}
\end{thm}
\begin{proof}
If $\Sigma_V=\emptyset$ and hence, $W=U$, then one proceeds exactly as in the proof of Thm.~\ref{thm:MCforRep l p different}, replacing the reference to \cite[Thm.~8.6]{Witte:MCVarFF} by \cite[Thm. 5.5]{Witte:UnitLFunctions}. For $\Sigma_V\neq \emptyset$ it remains to notice that $\id-\Frob_v q^{\deg(v)}$ is an automorphism of the finitely generated projective $\Lambda[[G]]$-module $(\Lambda[[G]]^\sharp\tensor_\Lambda T)^{\inertia_v}$ such that its class lies in $\KTh_1(\Lambda[[G]])\subset\KTh_1(\Lambda[[G]]_S)$ and hence, has trivial image under the boundary homomorphism $\bh$. Moreover,
\begin{multline*}
\eval_\rho([\id-\Frob_v q^{\deg(v)}\lcirclearrowright (\Lambda[[G]]^\sharp \tensor_\Lambda T)^{\inertia_v}])=\\
\det[\id-\gamma^{-1}\Frob_v q^{\deg(v)}\lcirclearrowright \Lambda[[\Gamma]]\tensor_\Lambda (T\tensor_\Lambda\rho)^{\inertia_v}]
\end{multline*}
such that
\begin{multline*}
\eval_\rho(\ncL_{K_{\infty}/K,\Sigma_W,\Sigma_V}(T))=\\
\eval_\rho\left(\ncL_{K_{\infty}/K,\Sigma_W,\emptyset}(T)\prod_{v\in\Sigma_V}[\id-\Frob_v q^{\deg(v)}\lcirclearrowright (\Lambda[[G]]^\sharp \tensor_\Lambda T)^{\inertia_v}]\right)\\
=L_{\Sigma_W,\Sigma_V}(T(\rho^\sharp),\gamma^{-1}).
\end{multline*}
\end{proof}

Note that in both cases, one can also allow $U$ to contain points $v$ in which $K_\infty/K$ has non-torsion ramification. However, by Lem.~\ref{lem:perfectness}, the corresponding non-commutative $L$-function $\ncL_{K_\infty/K,\Sigma_W,\Sigma_V}(T)$ then agrees with the $L$-function $\ncL_{K_\infty/K,\Sigma_W\cup\Sigma_0,\Sigma_V}(T)$ where the Euler factors in the set $\Sigma_0\subset U$ of points in which $K_\infty/K$ has non-torsion ramification are removed. The same Euler factors need then also be omitted in the interpolation property, i.\,e.\
$$
\eval_\rho(\ncL_{K_{\infty}/K,\Sigma_W,\Sigma_V}(T))=L_{\Sigma_W\cup \Sigma_0,\Sigma_V}(T(\rho^\sharp),\gamma^{-1}).
$$




\subsection{Duality and the functional equation}\label{ss:functional equation}

We retain the notation of the previous subsection. If $P$ is a finitely generated, projective $\Lambda[[G]]$-module, then
\[
P^{\Mdual}\coloneqq\Hom_{\Lambda[[G]]}(P,\Lambda[[G]])
\]
is a finitely generated, projective $\Lambda^\op[[G]]$-module if we let $g\in G$ act on $\phi\in P^{\Mdual}$ via
\[
g\phi\colon P\mto \Lambda[[G]],\qquad p\mapsto \phi(p)g^{-1}.
\]
The functor $\Mdual$ induces isomorphisms
\[
\begin{aligned}
\Mdual&\colon\KTh_1(\Lambda[[G]]_S)\mto\KTh_1(\Lambda^\op[[G]]_S),\\
\Mdual&\colon\KTh_0(\Lambda[[G]],S)\mto\KTh_0(\Lambda^\op[[G]],S)
\end{aligned}
\]
compatible with the boundary homomorphism of the localisation sequence \cite[Prop.~5.8]{Witte:BCP}.
If $\cmplx{P}$ is a strictly perfect complex of $\Lambda[[G]]$-modules and $f\colon\cmplx{P}\mto\cmplx{P}$ is an endomorphism whose cone is perfect as complex of $\Lambda[[H]]$-modules, then
\[
[f\lcirclearrowright \cmplx{P}]^\Mdual=[f^\Mdual\lcirclearrowright (\cmplx{P})^\Mdual]^{-1}.
\]
If $\cmplx{P}$ is a strictly perfect complex of $\Lambda[[G]]$-modules that is also perfect as complex of $\Lambda[[H]]$-modules, then
\[
[\cmplx{P}]^{\Mdual}=[(\cmplx{P})^{\Mdual}].
\]

The following lemma is a minor improvement on \cite[Prop. 5.4.17]{NSW:CohomNumFields}.

\begin{lem}\label{lem:duality passage to open subgroups}
Let $H'\subset H$ be an open subgroup and assume $M$ is a $\Lambda[[H]]$-module which has a resolution by finitely generated, free $\Lambda[[H']]$-modules. Then there exist an isomorphism
\[
\alpha\colon\RDer\Hom_{\Lambda[[H']]}(M,\Lambda[[H']])\xrightarrow{\isomorph} \RDer\Hom_{\Lambda[[H]]}(M,\Lambda[[H]])
\]
in the derived category of complexes of $\Lambda[[H']]^\op$-modules.
\end{lem}
\begin{proof}
Choose a system $g_1,\dots,g_d$ of right coset representatives of $H'\rquot H$. For any finitely generated, free $\Lambda[[H]]$-module $F$,
\[
\begin{aligned}
\alpha&\colon\Hom_{\Lambda[[H']]}(F,\Lambda[[H']])\mto \Hom_{\Lambda[[H]]}(F,\Lambda[[H]]),\\
\alpha(\phi)(f)&\coloneqq\sum_{i=1}^dg_i \phi(g_i^{-1}f)\qquad \text{for $f\in F$, $\phi\in\Hom_{\Lambda[[H']]}(F,\Lambda[[H']])$,}
\end{aligned}
\]
is an isomorphism of $\Lambda[[H']]^\op$-modules and does not depend on the choice of $g_1,\dots, g_d$. By Lem.~\ref{lem:finite presentations under finite ring extensions} below, $M$ also has a resolution $\cmplx{F}$ by finitely generated, free $\Lambda[[H]]$-modules. Moreover, any finitely generated, free $\Lambda[[H]]$-module is also finitely generated and free as $\Lambda[[H']]$-module, so that we can use the same $\cmplx{F}$ to compute the total derived functor of $\Hom_{\Lambda[[H]]}(M,\Lambda[[H]])$ in the categories of complexes of $\Lambda[[H]]$-modules or of $\Lambda[[H']]$-modules.
\end{proof}

\begin{lem}\label{lem:finite presentations under finite ring extensions}
Let $A$ be a subring of a ring $B$ and assume that $B$ has a resolution by finitely generated, free $A$-modules as a left $A$-module. Then a $B$-module $M$ has a resolution by finitely generated, free $B$-modules if and only if it has a resolution by finitely generated, free $A$-modules.
\end{lem}
\begin{proof}
If $M$ has a resolution $\cmplx{P}$ by finitely generated, free $B$-modules, then we may find a resolution of $P^{-n}$ by finitely generated, free $A$-modules for each $n\geq 0$. We obtain a resolution of $M$ by finitely generated, free $A$-modules by taking the total complex of the resulting double complex.

To prove the converse, we proceed by induction. For any ring $R$ and any $R$-module $N$, recall that a finite free presentation of length $\mu$ is an exact sequence
\[
F^{-\mu}\mto F^{1-\mu}\mto\cdots \mto F^0\mto N\mto 0
\]
with finitely generated, free $R$-modules $F^k$. Set $\lambda_R(N)\coloneqq-1$ if $N$ is not finitely generated and
\[
\lambda_R(N)\coloneqq\sup \set{\mu\given\text{there exists a finite free presentation of length $\mu$}}
\]
else. Clearly, for all $B$-modules $N$, if $\lambda_A(N)\geq 0$, then also $\lambda_B(N)\geq 0$. Assume that we know for some $n\geq 0$ that $\lambda_A(N)\geq n$ implies $\lambda_B(N)\geq n$ for $B$-modules $N$. Let $N'$ be a $B$-module with $\lambda_A(N')\geq n+1$. Then there exists an exact sequence
\[
0\mto Q\mto F\mto N'\mto 0
\]
of $B$-modules with $F$ finitely generated and free. By \cite[\S II.2, Ex.~6.(d)]{Bourbaki:CommAlg},
\[
\lambda_A(Q)\geq \inf\set{\lambda_A(F),\lambda_A(N')-1}\geq n.
\]
Hence, $\lambda_B(Q)\geq n$ by the induction assumption. By \cite[\S II.2, Ex.~6.(c)]{Bourbaki:CommAlg},
\[
\lambda_B(N')\geq \inf\set{\lambda_B(F),\lambda_B(Q)+1}\geq n+1.
\]
In particular, we conclude that $\lambda_A(M)=\infty$ implies $\lambda_B(M)=\infty$.
\end{proof}

\begin{lem}\label{lem:duality passage to H}
Assume that $M$ is a $\Lambda[[G]]$-module that has a resolution by finitely generated, projective $\Lambda[[H]]$-modules. Then $M$ also has a resolution by finitely generated, projective $\Lambda[[G]]$-modules and there exists an isomorphism
\[
\beta\colon\RDer\Hom_{\Lambda[[H]]}(M,\Lambda[[H]])\xrightarrow{\isomorph} \RDer\Hom_{\Lambda[[G]]}(M,\Lambda[[G]])[1]
\]
in the derived category of complexes of $\Lambda[[H]]^\op$-modules.
\end{lem}
\begin{proof}
By \cite[Lemma 3.11]{Witte:BCP} we may find a resolution $\cmplx{K}$ of $M$ by $\Lambda[[G]]$-modules which are finitely generated and projective as $\Lambda[[H]]$-modules. Choose a topological generator $\gamma\in\Gamma$. We then obtain an exact sequence of complexes of $\Lambda[[G]]$-modules
\[
0\mto \Lambda[[G]]\tensor_{\Lambda[[H]]}\cmplx{K}\xrightarrow{\id-(\cdot \gamma^{-1}\tensor\gamma\cdot)}\Lambda[[G]]\tensor_{\Lambda[[H]]}\cmplx{K}\mto\cmplx{K}\mto 0
\]
from \cite[Prop.~2.4]{Witte:Splitting}. The cone of $\id-(\cdot\gamma^{-1}\tensor\gamma\cdot)$ is a resolution of $M$ by finitely generated, projective $\Lambda[[G]]$-modules. One then proceeds as in \cite[Prop. 5.10]{Witte:BCP}.
\end{proof}

Assume that $M$ is a $\Lambda[[G]]$-module that has a resolution by finitely generated, projective $\Lambda[[H']]$-modules for some open subgroup $H'$ of $H$ and that
\[
\Ext^n_{\Lambda[[H'']]}(M,\Lambda[[H'']])=\Ext^{n+1}_{\Lambda[[G]]}(M,\Lambda[[G]])=0
\]
for all $n\neq 0$. By Lem.~\ref{lem:duality passage to open subgroups} and Lem.~\ref{lem:duality passage to H},
\[
\Ext^n_{\Lambda[[H'']]}(M,\Lambda[[H'']])=\Ext^{n+1}_{\Lambda[[G]]}(M,\Lambda[[G]])=0
\]
for all open subgroups $H''$ of $H$ and all $n\neq 0$. Let $\cg{\Ext^1_{\Lambda[[G]]}(M,\Lambda[[G]])}$ denote the $\Lambda[[G]]^\op$-module $\Ext^1_{\Lambda[[G]]}(M,\Lambda[[G]])$ considered as $\Lambda^\op[[G]]$-module.
\begin{defn}\label{defn:LambdaH-dual}
We write $\cg{M^{\mdual_{\Lambda[[H'']]}}}$ for $\Hom_{\Lambda[[H'']]}(M,\Lambda[[H']])$ considered as $\Lambda^\op[[G]]$-module via the isomorphism
\[
\cg{M^{\mdual_{\Lambda[[H'']]}}}\xrightarrow[\isomorph]{\alpha}
\cg{M^{\mdual_{\Lambda[[H]]}}}\xrightarrow[\isomorph]{\beta}
\cg{\Ext^1_{\Lambda[[G]]}(M,\Lambda[[G]])}.
\]
\end{defn}

If $H''$ is normal in $G$, $g\in G$ acts on $\phi\in\cg{M^{\mdual_{\Lambda[[H'']]}}}$ via
\[
g\phi\colon M\mto \Lambda[[H'']],\qquad m\mapsto g\phi(g^{-1}m)g^{-1}.
\]
In the case that $H''$ is not normal in $G$, it is more difficult to give an explicit description of the $G$-operation.

We conclude that if $M$ has a resolution by a strictly perfect complex of $\Lambda[[H]]$-modules, then
\begin{equation}\label{eqn:Mdual and LambdaH-dual}
[M]^{\Mdual}=-[\cg{M}^{\mdual_{\Lambda[[H']]}}]
\end{equation}
in $\KTh_0(\Lambda[[G]],S)$ for every open subgroup $H'$ of $H$.

Suppose now that $T$ is a finitely ramified $\Gal_K$-representation over $\Lambda$. We turn
\[
T^{\mdual_\Lambda}\coloneqq\Hom_{\Lambda}(T,\Lambda)
\]
into a finitely ramified $\Gal_K$-representation over $\Lambda^\op$ by letting $g\in \Gal_K$ act on $\phi\in T^{\mdual_\Lambda}$ via
\[
g\phi\colon T\mto\Lambda,\qquad t\mapsto \phi(g^{-1}t).
\]
Note that
\begin{align}
T^{\mdual_\Lambda}&\isomorph \dual{(\dual{\Lambda}\tensor_{\Lambda}T)},\label{eqn:mdual and dual}\\
(\Lambda[[G]]^\sharp\tensor_{\Lambda}T)^{\mdual_{\Lambda[[G]]}}&\isomorph\Lambda^\op[[G]]^\sharp\tensor_{\Lambda^\op}T^{\mdual_{\Lambda}}.\label{eqn:mdual and compact induction}
\end{align}

For any two \'etale sheaves $\sheaf{F}$, $\sheaf{G}$ of abelian groups on $C$, we will write $\sheafHom_C(\sheaf{F},\sheaf{G})$ for the sheaf of morphisms from $\sheaf{F}$ to $\sheaf{G}$. We will need the following corollary of the relative Poincar\'e duality theorem.

\begin{lem}\label{lem:exchange formula}
Assume that $\Lambda$ is a finite $\Int_{\ell}$-algebra with $\ell$ different from the characteristic of $K$. Then there exists a canonical isomorphism
\[
 \RDer \sheafHom_C(\RDer j_{V*} j_{U!}j_{K*}T,\Rat_\ell/\Int_\ell(1))\isomorph \RDer j_{W*}j_{U!}j_{K*}\dual{T}(1)
\]
in the derived category of complexes of \'etale sheaves of $\Lambda^\op$-modules on $C$.
\end{lem}
\begin{proof}
For any \'etale sheaf $\sheaf{F}$ of $\Lambda$-modules on $U$, the canonical morphism
$$
j_{V!}\RDer j_{U*}\sheaf{F}\mto \RDer j_{W*}j_{U!}\sheaf{F}
$$
is a isomorphism in the derived category of complexes of \'etale sheaves of $\Lambda$-modules on $C$. To see this, we note that
\begin{gather*}
j_{U}^*j_{V!}\RDer j_{U*}\sheaf{F}\isomorph\sheaf{F}\isomorph j_{U}^*\RDer j_{W*}j_{U!}\sheaf{F},\\
i_{\Sigma_V}^*\RDer j_{W*}j_{U!}\sheaf{F}\isomorph i_{\Sigma_V}^*j_{W}^*\RDer j_{W*}j_{U!}\sheaf{F}\isomorph 0 \isomorph i_{\Sigma_V}^*j_{V!}\RDer j_{U*}\sheaf{F},\\
i_{\Sigma_W}^*i_{\Sigma_{V}*}i_{\Sigma_{V}}^*\RDer j_{V*}\RDer j_{U*}\sheaf{F}\isomorph 0\isomorph i_{\Sigma_W}^*\RDer j_{W*}i_{\Sigma_{V}*}i_{\Sigma_{V}}^*\RDer j_{U*}\sheaf{F},\\
i_{\Sigma_W}^*j_{V!}\RDer j_{U*}\sheaf{F}\isomorph i_{\Sigma_W}^*\RDer j_{V*}\RDer j_{U*}\sheaf{F}\isomorph i_{\Sigma_W}^*\RDer j_{W*}\RDer j_{U*}\sheaf{F}\isomorph i_{\Sigma_W}^*\RDer j_{W*}j_{U!}\sheaf{F}.
\end{gather*}
Set
$$
D_U(\sheaf{F})\coloneqq\RDer\sheafHom_{U}(\sheaf{F},\Rat_\ell/\Int_\ell(1)).
$$
From the relative Poincar\'e duality theorem and the biduality theorem we obtain natural isomorphisms
$$
j_{U!}D_U(\sheaf{F})\isomorph D_V(\RDer j_{U*}\sheaf{F}),\qquad \RDer j_{U*}D_U(\sheaf{F})\isomorph D_V(j_{U!}\sheaf{F})
$$
in the derived category of complexes of \'etale sheaves of $\Lambda$-modules on $V$ \cite[Cor.~II.7.3, Thm~II.10.3, Cor.~II.7.5]{KiehlWeissauer:WeilConjectures}.
Finally we note that
$$
D_U(j_{K*}T)\isomorph j_{K*}\dual{T}(1)
$$
by \cite[Proof of Prop.~V.2.2.(b)]{Milne:EtCohom}.
\end{proof}

\begin{cor}\label{cor:Poincare duality with Pontryagin dual}
Assume that $\Lambda$ is a finite $\Int_{\ell}$-algebra with $\ell$ different from the characteristic of $K$. Then there exists canonical isomorphisms
\[
 \begin{aligned}
 \RDer\Hom_{\Int}(\RDer \Sect(\algc{V},j_{U!}j_{K*}\dual{T}(1)),\Rat_\ell/\Int_\ell)&\isomorph \RDer\Sect(\algc{W},j_{U!}j_{K*}T)[2],\\
  \RDer\Hom_{\Int}(\RDer \Sect(V,j_{U!}j_{K*}\dual{T}(1)),\Rat_\ell/\Int_\ell)&\isomorph \RDer\Sect(W,j_{U!}j_{K*}T)[3]
 \end{aligned}
\]
in the derived category of complexes of $\Lambda$-modules. The first isomorphism is compatible with the operations of $\Frob_{\FF}$ on the left-hand complex and $\Frob_{\FF}^{-1}$ on the righthand complex.
\end{cor}
\begin{proof}
Apply the relative Poincar\'e duality theorem to the structure morphism $C\mto\Spec \FF$ and the complex of \'etale sheaves $\RDer j_{V*}j_{U!}j_{K_*}\dual{T}(1)$ on $C$. Then take global sections over $\Spec \algc{\FF}$ and $\Spec \FF$. Finally, use Lem.~\ref{lem:exchange formula} to obtain identifications
\[
\begin{aligned}
\RDer\Hom_{\algc{C}}(\RDer j_{V*}j_{U!}j_{K*}\dual{T}(1),\Rat_{\ell}/\Int_{\ell}(1))&\isomorph\RDer\Sect(\algc{W},j_{U!}j_{K*}T),\\
\RDer\Hom_{C}(\RDer j_{V*}j_{U!}j_{K*}\dual{T}(1),\Rat_{\ell}/\Int_{\ell}(1))&\isomorph\RDer\Sect(W,j_{U!}j_{K*}T).
\end{aligned}
\]
\end{proof}

Let now $\Lambda$ be a general adic $\Int_\ell$-algebra. In the following corollary, we consider a complex $\cmplx{P}=(\cmplx{P}_I)_{I\in\openideals_{\Lambda}}$ of $\cat{PDG}^{\cont}(\Lambda)$ as objects of the derived category of complexes of $\Lambda$-modules by passing to the projective limit
\[
 \varprojlim_{I\in\openideals_{\Lambda}}\cmplx{P}_I.
\]
We recall that the projective limit is an exact functor by the construction of $\cat{PDG}^{\cont}(\Lambda)$.

\comment{To do: Check this}
\begin{cor}\label{cor:Poincare duality with Lambda dual}
Assume that $\Lambda$ is an adic $\Int_{\ell}$-algebra with $\ell$ different from the characteristic of $K$ and that $T$ has projective local cohomology over $U$. Then there exists canonical isomorphisms
\[
 \begin{aligned}
 \RDer\Hom_{\Lambda^\op}(\RDer \Sect(\algc{V},j_{U!}j_{K*}T^{\mdual_{\Lambda}}(1)),\Lambda^\op)&\isomorph \RDer\Sect(\algc{W},j_{U!}j_{K*}T)[2],\\
  \RDer\Hom_{\Lambda^\op}(\RDer \Sect(V,j_{U!}j_{K*}T^{\mdual_{\Lambda}}(1)),\Lambda^\op)&\isomorph \RDer\Sect(W,j_{U!}j_{K*}T)[3]
 \end{aligned}
\]
in the derived category of complexes of $\Lambda$-modules.
\end{cor}
\begin{proof}
For any finitely generated, projective $\Lambda$-module $P$, we have
\[
\dual{(\dual{\Lambda}\tensor_{\Lambda}P)}\isomorph\Hom_{\Lambda}(P,\Lambda)
\]
by the adjunction formula for $\Hom$ and $\tensor$ and by recalling that every homomorphism from $P$ to $\Lambda$ is automatically continuous for the compact topology. Hence,
\[
 \RDer\Hom_{\Lambda^\op}(\cmplx{P},\Lambda^\op)\isomorph\RDer\Hom_{\Int}(\dual{\Lambda^\op}\Ltensor_{\Lambda^\op}\cmplx{P},\Rat_\ell/\Int_{\ell})
\]
for every perfect complex of $\Lambda^\op$-modules $\cmplx{P}$. Further,
\[
 \dual{(\Lambda^\op)}\Ltensor_{\Lambda^\op}\RDer\Sect(V,j_{U!}\cmplx{\sheaf{F}})\isomorph
 \varinjlim_{I\in\openideals_{\Lambda}}\RDer\Sect(V,j_{U!}\dual{(\Lambda^\op/I)}\tensor_{\Lambda^\op/I}\cmplx{\sheaf{F}}_I)
\]
for any $\cmplx{\sheaf{F}}$ in $\cat{PDG}^\cont(U,\Lambda^\op)$. The same is true for $V$ replaced by $\algc{V}$. Arguing as in Lem.~\ref{lem:commutativity with tensor product}, we further see that the natural morphism
\[
 \dual{(\Lambda^\op/I)}\tensor_{\Lambda^\op/I}(j_{K*}T^{\mdual_\Lambda}(1))_I\mto j_{K*}\dual{(\Lambda^\op/I)}\tensor_{\Lambda^\op}T^{\mdual_\Lambda}(1)
\]
is an isomorphism of \'etale sheaves on $U$. We now apply Cor.~\ref{cor:Poincare duality with Pontryagin dual} to the $\Gal_K$-representation $T/IT$ over $\Lambda/I$ for each $I$ in $\openideals_{\Lambda}$ and observe that
\[
\dual{(T/IT)}(1)\isomorph \dual{(\Lambda^\op/I)}\tensor_{\Lambda^\op}T^{\mdual_\Lambda}(1).
\]
by identity~\eqref{eqn:mdual and dual}.
\end{proof}

\begin{rem}
Even for finite $\Lambda$, Cor.~\ref{cor:Poincare duality with Lambda dual} is wrong without the hypothesis that $T$ has projective local cohomology over $U$. If one merely assumes that $T$ and $T^{\mdual_\Lambda}(1)$ have projective stalks over $U$, then the cone of the natural duality morphism
\[
 \RDer\Sect(\algc{W},j_{U!}\eta_{*}T)[2]\mto \RDer\Hom_{\Lambda^\op}(\RDer \Sect(\algc{V},j_{U!}\eta_{*}T^{\mdual_{\Lambda}}(1)),\Lambda^\op)
\]
is given by the complex
\[
\cmplx{C}\colon\qquad \bigoplus_{v\in U^0}\HF^1(\inertia_v,T)\mto \bigoplus_{v\in U^0}(T^{\mdual_\Lambda}(1)^{\inertia_v})^{\mdual_\Lambda}
\]
sitting in degrees $-1$ and $0$, with
\begin{align*}
\HF^{-1}(\cmplx{C})&\isomorph\bigoplus_{v\in U^0}\Ext^1_{\Lambda^\op}(\HF^1(\inertia_v,T^{\mdual_\Lambda}(1)),\Lambda^\op)\\
\HF^0(\cmplx{C})&\isomorph\bigoplus_{v\in U^0}\Ext^2_{\Lambda^\op}(\HF^1(\inertia_v,T^{\mdual_\Lambda}(1)),\Lambda^\op).
\end{align*}
Moreover, if $T$ has projective stalks in $v$, the same does not need to be true for $T^{\mdual_\Lambda}(1)$, and vice versa. For example, the dual of the representation $T$ from Rem.~\ref{rem:non fg stalk} satisfies $(T^{\mdual_\Lambda}(1))^{\inertia_v}=0$ for the given $v$.
\end{rem}


\begin{defn}
Assume $\ell\neq p$. Let $T$ be a finitely ramified $\Gal_K$-representation over $\Lambda$ with projective stalks over $U$. The global $\epsfact$-factor of
$\RDer j_{U*}j_{K*}T$ on $W$ is given by
\[
\begin{aligned}
\epsfact(W,\RDer j_{U*}j_{K*}T)&=[-\Frob_{\FF}\lcirclearrowright \RDer\Sect(\algc{V},j_{U!}j_{K*}T)]\in\KTh_1(\Lambda).
\end{aligned}
\]
If $K_\infty/K$ is an admissible extension such that $\Lambda[[G]]^\sharp\tensor T$ has projective stalks over $U$, we set
\[
 \epsfact_{K_\infty/K,\Sigma_W,\Sigma_V}(T)\coloneqq \epsfact(W,\RDer j_{U*}j_{K*}(\Lambda[[G]]^\sharp\tensor_{\Lambda}T)).
\]
\end{defn}

\begin{rem}
It is expected that the global $\epsfact$-factor may be expressed as a product of local $\epsfact$-factors. For $\Lambda=\Int_{\ell}$, this is a theorem of Laumon \cite[Thm.~3.2.1.1]{Laumon:TransfDeFourier+ConstdEF+WC}. In \cite[\S 3.5.6]{FK:CNCIT}, Fukaya and Kato sketch how to extend this result to arbitrary adic $\Int_{\ell}$-algebras.
\end{rem}

We obtain the following functional equation for $\ncL_{K_\infty/K,\Sigma_W,\Sigma_V}(T)$.

\comment{To do: Check this! Add trivial case}
\begin{thm}\label{thm:functional equation}
Assume $\ell\neq p$. Let $K_\infty/K$ be an admissible extension and $T$ be a finitely ramified $\Gal_K$-representation over $\Lambda$. Assume that $K_\infty/K$ has ramification prime to $\ell$  and $T$ has projective local cohomology over $K_\infty$ in all closed points of $U$. Then
\[
(\ncL_{K_\infty/K,\Sigma_V,\Sigma_W}(T^{\mdual_\Lambda}(1)))^{\Mdual}\ncL_{K_\infty/K,\Sigma_W,\Sigma_V}(T)\epsfact_{K_\infty/K,\Sigma_W,\Sigma_V}(T)=1.
\]
\end{thm}
\begin{proof}
Choose a strictly perfect complex $\cmplx{P}$ of $\Lambda^\op[[G]]$-modules, an endomorphism $f\colon \cmplx{P}\mto \cmplx{P}$, and a weak equivalence
\[
\alpha\colon \cmplx{P}\mto \RDer\Sect(\algc{W},j_{U!}j_{K*}\Lambda^\op[[G]]^\sharp\tensor_{\Lambda^\op} T^{\mdual_\Lambda}(1))
\]
such that the diagram
\[
 \xymatrix{
 \cmplx{P}\ar[r]^-{\alpha}\ar[d]^{\id-f}& \RDer\Sect(\algc{W},j_{U!}j_{K*}\Lambda^\op[[G]]^\sharp\tensor_{\Lambda^\op} T^{\mdual_\Lambda}(1))\ar[d]^{\id-\Frob_{\FF}}\\
 \cmplx{P}\ar[r]^-{\alpha}              & \RDer\Sect(\algc{W},j_{U!}j_{K*}\Lambda^\op[[G]]^\sharp\tensor_{\Lambda^\op} T^{\mdual_\Lambda}(1))
 }
\]
commutes in the derived category of complexes of $\Lambda^\op[[G]]$-modules. In particular, the diagram commutes up to homotopy in the Waldhausen category of perfect complexes of $\Lambda^\op[[G]]$-modules.
By \cite[Lem.~3.1.6]{Witte:PhD}, this implies
\[
 [\id-f\lcirclearrowright \cmplx{P}]^{-1}=\ncL_{K_\infty/K,\Sigma_V,\Sigma_W}(T^{\mdual_\Lambda}(1))
\]
in $\KTh_1(\Lambda^\op[[G]],S)$. Applying Cor.~\ref{cor:Poincare duality with Lambda dual} to the representation $\Lambda[[G]]^\sharp\tensor_{\Lambda} T$ over $\Lambda[[G]]$ and using \eqref{eqn:mdual and compact induction}, we obtain a commutative diagram
\[
 \xymatrix{
 (\cmplx{P})^\Mdual\ar[r]^-{\beta}\ar[d]^{\id-f^\Mdual}& \RDer\Sect(\algc{V},j_{U!}j_{K*}\Lambda[[G]]^\sharp\tensor_{\Lambda} T)[2]\ar[d]^{\id-\Frob_{\FF}^{-1}}\\
 (\cmplx{P})^\Mdual\ar[r]^-{\beta}              & \RDer\Sect(\algc{V},j_{U!}j_{K*}\Lambda[[G]]^\sharp\tensor_{\Lambda} T)[2]
 }
\]
in the derived category of complexes of $\Lambda[[G]]$-modules. Hence,
\[
 \begin{aligned}
  \ncL_{K_\infty/K,\Sigma_V,\Sigma_W}(T^{\mdual_\Lambda}(1))^\Mdual&=[\id-f^\Mdual\lcirclearrowright (\cmplx{P})^\Mdual]\\
                                                                    &=[\id-\Frob_{\FF}^{-1}\lcirclearrowright\RDer\Sect(\algc{V},j_{U!}j_{K*}\Lambda[[G]]^\sharp\tensor_{\Lambda} T)]\\                                                                   &=\epsfact_{K_\infty/K,\Sigma_W,\Sigma_V}(T)^{-1}\ncL_{K_\infty/K,\Sigma_W,\Sigma_V}(T)^{-1}
 \end{aligned}
\]
in $\KTh_1(\Lambda[[G]],S)$
\end{proof}

\subsection{Calculation of the cohomology}

In this section, we will give a description of the cohomology groups of the complex $\RDer\Sect(V,j_{U!}j_{K*}(\Lambda[[G]]^\sharp\tensor_\Lambda T))$.

\begin{lem}\label{lem:calculation of cohomology}
We assume either $\ell\neq p$ or $V=C$. We further assume that $(\Lambda[[G]]^\sharp\tensor_\Lambda T)^{\inertia_v}$ is finitely generated for each closed point of $U$. Then
\begin{enumerate}
 \item $\HF^k(V,j_{U!}j_{K*}(\Lambda[[G]]^\sharp\tensor_\Lambda T))=0$ for $k\notin\set{1,2,3}$.
 \item $$
       \HF^1(V,j_{U!}j_{K*}(\Lambda[[G]]^\sharp\tensor_\Lambda T))=\begin{cases}
                                                      T^{\Gal_{K_\infty}}&\text{if $U=V$ and $H$ is finite,}\\
                                                      0&\text{else.}
                                                     \end{cases}
       $$
 \item If $\ell\neq p$, then
       $$
       \HF^2(V,j_{U!}j_{K*}(\Lambda[[G]]^\sharp\tensor_\Lambda T))=\dual{\HF^1(W_{K_{\infty}},j_{U_{K_\infty}!}j_{K_{\infty}*}\dual{T}(1))}.
       $$
 \item $$
       \HF^3(V,j_{U!}j_{K*}(\Lambda[[G]]^\sharp\tensor_\Lambda T))=\begin{cases}
                                                      T(-1)_{\Gal_{K_\infty}}&\text{if $\ell\neq p$ and $V=C$,}\\
                                                      0&\text{else.}
                                                     \end{cases}
       $$
\end{enumerate}
\end{lem}
\begin{proof}
In the view of Prop.~\ref{prop:limit property}  we may assume that $\Lambda$ is a finite ring. We will first consider the case that $H$ is finite. Using Lem.~\ref{lem:HS argument} and the fact that the cohomology of an \'etale sheaf of $\Lambda$-modules on the curve $V_{\algc{\FF}K_{\infty}}$ over the algebraically closed field $\algc{\FF}$ is concentrated in degrees $0$ up to $2$ if $\ell\neq p$ and $V=C$ and up to $1$ if $\ell=p$ \cite[Cor.~VI.2.5]{Milne:EtCohom} or $V\neq C$  \cite[Rem.~V.2.4]{Milne:EtCohom} we deduce Assertion $(1)$ and the second case of Assertion $(4)$. Assertion $(2)$ for $H$ finite follows since
$$
 \HF^0(V_{K_{\infty}},j_{U_{K_\infty}!}j_{K_{\infty}*}T)=\begin{cases}
                                          T^{\Gal_{K_\infty}}&\text{if $U=V$}\\
                                          0&\text{else.}
                                           \end{cases}
$$

We now assume $\ell\neq p$. Assertion $(3)$ is a consequence of Cor.~\ref{cor:Poincare duality with Pontryagin dual}. Moreover, this corollary implies
\begin{align*}
\HF^3(C,j_{U!}j_{K*}(\Lambda[[G]]^\sharp\tensor_\Lambda T))& =\dual{\HF^0(U_{K_{\infty}},j_{K_{\infty}*}\dual{T}(1))}\\
                                &=\dual{((\dual{T}(1))^{\Gal_{K_\infty}})}=T(-1)_{\Gal_{K_\infty}}.
\end{align*}
This proves Assertion $(4)$ in the case $\ell\neq p$.

Finally, we use Cor.~\ref{cor:limit porperty for admissible extension} to deduce the assertions for general $H$. In the case of Assertion $(2)$ it remains to notice that, since $T$ is finite, there exists a finite extension $L/K_\cyc$ inside $K_\infty$ with $T=T^{\Gal_L}$ and such that $\Gal(K_\infty/L)$ is pro-$\ell$. Hence, the norm map $N_{L''/L'}\colon T\mto T$ is multiplication by a power of $\ell$ for $L\subset_f L'\subset_f L''\subset K_\infty$. We conclude that
$$
\HF^1(V,j_{K*}(\Lambda[[G]]^\sharp\tensor_{\Lambda} T))=\varprojlim_{K_\cyc\subset_f L\subset K_\infty}T^{\Gal_L}=0
$$
if $H$ is infinite.
\end{proof}

As we will explain in Section~\ref{ss:Picard-1-motives}, the following two corollaries may be viewed as a vast generalisation of the main result of \cite{GreitherPopescu:PicardOneMotives}.

\comment{Correct this! Continue here!}
\begin{cor}\label{cor:projectivity ell neq pee}
For $\ell\neq p$, assume
\begin{enumerate}
 \item[(i)] $K_\infty/K$ has ramification prime to $\ell$ in all closed points of $U$,
 \item[(ii)] $T$ has projective local cohomology over $K_\infty$ in all closed points of $U$,
 \item[(iii)] either $U\neq V$ or $(T^{\mdual_\Lambda})_{\Gal_{K_\infty}}=0$, and
 \item[(iv)] either $V\neq C$ or $T(-1)_{\Gal_{K_\infty}}=0$.
\end{enumerate}
Then:
\begin{enumerate}
\item The $\Lambda[[H]]$-module $\HF^2(V,j_{U!}j_{K*}(\Lambda[[G]]^\sharp\tensor_\Lambda T))$ is finitely generated and projective and
$$
[\HF^2(V,j_{U!}j_{K*}(\Lambda[[G]]^\sharp\tensor_\Lambda T))]=[\RDer\Sect(V,j_{U!}j_{K*}(\Lambda[[G]]^\sharp\tensor_\Lambda T))]
$$
in $\KTh_0(\Lambda[[G]],S)$.
\item
\begin{multline*}
[\HF^2(V,j_{U!}j_{K*}(\Lambda[[G]]^\sharp\tensor_\Lambda T))]=\\
[\cg{\HF^2(W,j_{U!}j_{K*}(\Lambda^\op[[G]]^\sharp\tensor_{\Lambda^\op}T^{\mdual_\Lambda})(1))^{\mdual_{\Lambda^\op[[H]]}}}]\\
=-[\HF^2(W,j_{U!}j_{K*}(\Lambda^\op[[G]]^\sharp\tensor_{\Lambda^\op}T^{\mdual_\Lambda})(1))]^\Mdual
\end{multline*}
in $\KTh_0(\Lambda[[G]],S)$.
\end{enumerate}
\end{cor}
\begin{proof}
Note that
$
T^{\Gal_{K_\infty}}\isomorph\Hom_{\Lambda^\op}((T^{\mdual_\Lambda})_{\Gal_{K_\infty}},\Lambda^\op).
$
By Lem.~\ref{lem:calculation of cohomology} and assumptions (iii) and (iv), $\HF^2(V,j_{U!}j_{K*}(\Lambda[[G]]^\sharp\tensor_\Lambda T))$ is the only non-vanishing cohomology group of $\RDer\Sect(V,j_{U!}j_{K*}(\Lambda[[G]]^\sharp\tensor_\Lambda T))$.

Let $\rho$ be a simple $\Lambda[[G]]^{\op}$-module. Then $\rho$ is also simple as module over $(\Lambda[[G]]/\Jac(\Lambda[[G]]))^{\op}$. By Schur's lemma the endomorphism ring of $\rho$ is division ring $k$. Since $k$ is clearly finite, it is a field. Hence, we may consider $\rho$ as $k$-$\Lambda[[G]]$-bimodule, which is finitely generated and projective as $k$-module. If $T$ has projective local cohomology over $K_\infty$ in all closed points of $U$, then $T(\rho^\sharp)$ has projective local cohomology over $U$ and the natural map
\[
\ringtransf_{\rho}\RDer\Sect(V,j_{U!}j_{K*}(\Lambda[[G]]^\sharp\tensor_\Lambda T))\wto
\RDer\Sect(V,j_{U!}j_{K*}T(\rho^\sharp))
\]
is a weak equivalence by Lem.~\ref{lem:commutativity with tensor product}. Moreover, $T^{\mdual_\Lambda}(1)$ also has projective local cohomology over $K_\infty$ in all closed points of $U$. If $U\neq V$, then
\[
 \HF^0(V,j_{U!}j_{K*}T(\rho^\sharp))\isomorph\ker(\HF^0(V,j_{U*}j_{K*}T(\rho^\sharp))\xrightarrow{\isomorph}\HF^0(U,j_{K*}T(\rho^\sharp)))=0
\]
by the Leray spectral sequence. If $U=V$ and $(T^{\mdual_\Lambda})_{\Gal_{K_\infty}}=0$, then
\[
 \HF^0(U,j_{K*}T(\rho^\sharp))\isomorph T(\rho^\sharp)^{\Gal_{K_\infty}}\isomorph(((T(\rho^\sharp))^{\mdual_k})_{\Gal_{K_\infty}})^{\mdual_k}=0.
\]
In particular, the flat dimension of the $\Lambda[[G]]$-module $\HF^2(V,j_{U!}j_{K*}(\Lambda[[G]]^\sharp\tensor_{\Lambda} T))$ is less or equal to $1$ in both cases.

By Thm.~\ref{thm:MCforRep l p different} the complex $\RDer\Sect(V,j_{U!}j_{K*}(\Lambda[[G]]^\sharp\tensor_\Lambda T))$ is an object of the Waldhausen category $\cat{PDG}^{\cont,w_H}(\Lambda[[G]])$. In particular, we may find strictly perfect complexes $\cmplx{P}$ and $\cmplx{Q}$ of $\Lambda[[G]]$-modules and $\Lambda[[H]]$-modules, respectively, which are weakly equivalent to $\RDer\Sect(V,j_{U!}j_{K*}(\Lambda[[G]]^\sharp\tensor_\Lambda T))$. By the above observations, we conclude that we may choose $\cmplx{P}$ to be concentrated in degrees $1$ and $2$. Hence, $\HF^2(V,j_{U!}j_{K*}(\Lambda[[G]]^\sharp\tensor_{\Lambda} T))$ is finitely generated and projective over $\Lambda[[H]]$ by \cite[Lem.~3.4]{Witte:BCP}. We then use Rem.~\ref{rem:class of a projective RH module} to complete the proof of (1).

The same reasoning applies to $T^{\mdual_\Lambda}(1)$.  We now apply Cor.~\ref{cor:Poincare duality with Lambda dual} to the $\Gal_K$-representation $\Lambda[[G]]^\sharp\tensor_{\Lambda}T$ over $\Lambda[[G]]$ and use the identity \eqref{eqn:Mdual and LambdaH-dual} to deduce (2).

\end{proof}

\begin{cor}\label{cor:projectivity ell eq pee}
For $\ell=p$, assume
\begin{enumerate}
 \item[(i)] $K_\infty/K$ has ramification prime to $p$ over $U$,
 \item[(ii)] $T$ has ramification prime to $p$ over $U$, and
 \item[(iii)] either $U\neq X$ or $(T^{\mdual_\Lambda})_{\Gal_{K_\infty}}=0$.
\end{enumerate}
Then the $\Lambda[[H]]$-module $\HF^2(X,j_{U!}j_{K*}(\Lambda[[G]]^\sharp\tensor_\Lambda T))$ is finitely generated and projective and
$$
[\HF^2(X,j_{U!}j_{K*}(\Lambda[[G]]^\sharp\tensor_\Lambda T))]=[\RDer\Sect(X,j_{U!}j_{K*}(\Lambda[[G]]^\sharp\tensor_\Lambda T))]
$$
in $\KTh_0(\Lambda[[G]],S)$.
\end{cor}
\begin{proof}
We proceed as in the proof of Cor.~\ref{cor:projectivity ell neq pee}.(1), using Theorem~\ref{thm:MCforRep l equal p myLfunc}.
\end{proof}

\subsection{The main conjecture for Selmer groups}

In this section we will assume that $R$ is a local, commutative, and regular adic $\Int_\ell$-algebra with $\ell\neq p$. Furthermore, we assume that $\Sigma_V=\emptyset$, such that $U=W$ and $V=C$. Let $T$ be a finitely ramified representation of $\Gal_K$ over $R$. For $K\subset L\subset \algc{K}$ we may define as in \cite[\S 5]{Greenberg:IwThForReps} a Selmer group
$$
\Sel_{\Sigma_W}(L,\dual{T}(1))\coloneqq\ker\left(\HF^1(\Gal_L,\dual{T}(1))\mto \bigoplus_{v\in W^0_L} \HF^1(\inertia_v,\dual{T}(1))\right),\\
$$
for $T$.

If $K_\infty/K$ is an admissible $\ell$-adic Lie extension we set
$$
\X_{K_\infty/K,\Sigma_W}(T)\coloneqq\dual{\Sel_{\Sigma_W}(K_\infty,\dual{T}(1))}.
$$

\begin{lem}\label{lem:selmer group}
For any extension $L/K$ inside $\algc{K}$,
$$
\Sel_{\Sigma_W}(L,\dual{T}(1))=\HF^1(W_L,j_{L*}\dual{T}(1)).
$$
In particular,
$$
\HF^2(C,j_{W!}j_{K*}(R[[G]]^\sharp\tensor_R T))=\X_{K_\infty/K,\Sigma_W}(T)
$$
for any admissible $\ell$-adic Lie extension $K_\infty/K$.
\end{lem}
\begin{proof}
Without loss of generality we assume that $L=K$. According to \cite[Lem.~III.1.16]{Milne:EtCohom} we have for every integer $k$
$$
\HF^k(\Gal_K,\dual{T}(1))=\varinjlim_{U}\HF^k(U,j_{K*}\dual{T}(1)).
$$
Here, $U$ runs through the open dense subschemes of $W$. For any such $U$, the Leray spectral sequence shows
$$
\HF^1(W,j_{K*}\dual{T}(1))=\ker \HF^1(U,j_{K*}\dual{T}(1))\mto\bigoplus_{v\in W-U}\HF^0(v,i_{v}^*\RDer^1j_{K*}\dual{T}(1)).
$$
Recall that for any discrete $\Gal_K$-module $M$ one has $(i_v^*j_{K*}M)_{\xi_v}=M^{\inertia_v}$. By considering an injective resolution of $\dual{T}(1)$ we conclude
$$
(i_{v}^*\RDer^1j_{K*}\dual{T}(1))_{\xi_v}=\HF^1(\inertia_v,\dual{T}(1)).
$$
The first equality in the lemma follows after passing to the direct limit over $U$. The second equality follows from the first equality and Lem.~\ref{lem:calculation of cohomology}.
\end{proof}

We may thus deduce the following reformulation of the first part of the non-commutative main conjecture in terms of the $R[[G]]$-module $\X_{K_\infty/K,\Sigma_W}(T)$.

\begin{cor}\label{cor:mc for selmer groups}
Assume that $G$ has no element of order $\ell$. Let $\Sigma_0\subset W$ denote the set of points over which $K_\infty/K$ has non-torsion ramification and assume that $T$ has projective stalks over $K_\infty$ in all closed points of $W-\Sigma_0$. Then
\begin{enumerate}
\item $\X_{K_\infty/K,\Sigma_W}(T)$ is in $\cat{N}_H(R[[G]])$.
\item In $\KTh_0(R[[G]],S)$ we have
      \begin{align*}
      \bh \ncL_{K_\infty/K,\Sigma_W\cup\Sigma_0,\emptyset}(T)=&-[\X_{K_\infty/K,\Sigma_W}(T)]+[T(-1)_{\Gal_{K_\infty}}]\\
      &+
      \begin{cases}
      [T^{\Gal_{K_\infty}}]&\text{if $\Sigma_W=\emptyset$ and $H$ is finite,}\\
      0&\text{else.}
      \end{cases}
      \end{align*}
\end{enumerate}
\end{cor}
\begin{proof}
Note that
\[
 \X_{K_\infty/K,\Sigma_W}(T)=\X_{K_\infty/K,\Sigma_W\cup\Sigma_0}(T)
\]
by Lem.~\ref{lem:perfectness}. The rest is a direct consequence of \eqref{eq:NHG computes relative K-group}, Thm.~\ref{thm:MCforRep l p different}, and Lem.~\ref{lem:calculation of cohomology}.
\end{proof}

\begin{rem}
We recall from Example~\ref{exmpl:ramification types} that if $R$ has global dimension less or equal to $2$, for example $R=\Int_\ell$ or $R=\Int_\ell[[t]]$, then $T$ has automatically projective stalks over $K_\infty$ in every closed point of $C$.
\end{rem}

Quite often, the class $[T(-1)_{\Gal_{K_\infty}}]$ is zero in $\KTh_0(R[[G]],S)$. However, this is not always the case. Since the forgetful functor from $\cat{N}_H(R[[G]])$ to the category of finitely generated $R[[H]]$-modules induces a homomorphism $\KTh_0(R[[G]],S)\mto\KTh_0(R[[H]])$, a necessary condition is that $[T(-1)_{\Gal_{K_\infty}}]$ is zero in $\KTh_0(R[[H]])$. For this condition, we can formulate the following useful criterion, which is essentially due to Serre (see also \cite[\S 1.3]{ArdakovWadsley:DimensionFiltration}). In particular, we see that this condition is not satisfied by the group $H=\Int_{\ell}^d\rtimes \mu_{\ell-1}$ with the group of $\ell-1$-th roots of units acting by multiplication on $\Int_{\ell}^d$ if $\ell>2$.

Recall that a $\ell$-adic Lie group $H$ is called \emph{virtually solvable} if its Lie algebra $L(H)$ is solvable.

\begin{lem}\label{lem:vanishing in K0H}
Let $H$ be a compact $\ell$-adic Lie group without any element of order $\ell$ and $R$ a commutative, local, regular adic $\Int_\ell$-algebra. The class $[T]$ of every $R[[H]]$-module $T$ which is finitely generated as $R$-module is zero in $\KTh_0(R[[H]])$ precisely if the centraliser of every element of finite order in $H$ has infinitely many elements. This condition is satisfied if $H$ is a pro-$\ell$-group or if $H$ is not virtually solvable.
\end{lem}
\begin{proof}
By the Cohen Structure Theorem \cite[Ch. VIII, \S 5, Thm.~2]{Bourbaki:CommAlg}, we have
\[
R\isomorph R_0[[X_1,\dots,X_n]],
\]
where $R_0$ is either a finite field of characteristic $\ell$ or the valuation ring of a finite field extension of $\Rat_\ell$ and $X_1,\dots, X_n$ are indeterminates.  We do induction on $n$.

Assume that $n=0$ and that $R_0$ is a finite field. By \cite[Cor. to Thm. C]{Serre:EulerPoincare}  the Euler characteristic of every $R_0[[H]]$-module $T$ which is of finite dimension over $R_0$ is trivial precisely if the centraliser of every element of $H$ has infinitely many elements. For any element of infinite order this is clearly an empty condition. Now the proof of \cite[Thm.~8.2, $(a)\Rightarrow(b)$]{ArdakovWadsely:CharElements} shows that the vanishing of the Euler characteristics is equivalent with the vanishing of the classes $[T]$.

If $H$ is a pro-$\ell$-group without any element of order $\ell$, then there are no elements of finite order at all. If $H$ is not virtually solvable, then its Lie algebra $L(H)$ is not solvable. Since any element $h\in H$ of finite order has order prime to $\ell$, the image of $h$ in the automorphism group of $L(H)$ must be semi-simple. By an old result of Borel and Mostow \cite[Thm.~4.5]{Borel:SemiSimpleAutos} (the author thanks S.~Wadsley for pointing out this reference to him), $h$ fixes a non-trivial subspace of $L(H)$, which implies that the centraliser of $h$ in $H$ must be infinite.

Now assume that $R_0$ is the valuation ring of a finite field extension of $\Rat_\ell$. Let $\pi\in R_0$ be a uniformiser and $k$ the residue field of $R_0$. By Quillen's devissage theorem \cite[Thm.~4]{Quillen:HigherKTheory}, we may identify $\KTh_0(k[[H]])$ with the $\KTh$-group of the abelian category of finitely generated $R_0[[H]]$-modules that are annihilated by a power of $\pi$. Under this identification, classes of those $R_0[[H]]$-modules which are additionally finitely generated over $R_0$ are mapped to the subgroup of $\KTh_0(k[[H]])$ generated by the classes of those $k[[H]]$-modules that are finitely generated over $k$.  By Quillen's localisation theorem \cite[Thm.~5]{Quillen:HigherKTheory} we thus obtain an exact sequence
\[
\KTh_0(k[[H]])\mto\KTh_0(R_0[[H]])\mto\KTh_0(R_0[[H]][\frac{1}{\pi}])\mto 0,
\]
noting that all rings in this sequence are of finite global dimension. For any  $R_0[[H]]$-module $T$ which is finitely generated over $R_0$, there exists an exact sequence of $R_0[[H]]$-modules, finitely generated over $R_0$,
\[
0\mto T'\mto T\mto T''\mto 0,
\]
where $T'$ is annihilated by a power of $\pi$ and $\pi$ is a non-zero divisor on $T''$.

Assume that the centraliser of every element of finite order in $H$ has infinitely many elements. We already know that $[T']=0$. Hence, we may assume that $\pi$ is a non-zero divisor on $T$. In particular,
\[
k[[H]]\tensor_{R_0[[H]]}T=k\tensor_{R_0}T
\]
agrees with the derived tensor product with $k[[H]]$ over $R_0[[H]]$ and is finitely generated over $k$. Hence, $[k\tensor_{R_0}T]=0$ in $\KTh_0(k[[H]])$. Since $\pi$ is in the Jacobson radical of $R_0[[H]]$, the derived tensor product with $k[[H]]$ induces an isomorphism
\[
\KTh_0(R_0[[H]])\mto\KTh_0(k[[H]]).
\]
Hence $[T]=0$ in $\KTh_0(R_0[[H]])$.

Conversely, if $H$ does not satisfy the above property, we may find a $k[[H]]$-module $T$ which is finitely generated over $k$ and which has non-trivial class $[T]$ in $\KTh_0(k[[H]])$. The image of $H$ in the automorphism group of $T$ is a finite group $\Delta$. By \cite[Thm. 33]{Serre:LRFG}, $[T]$ has a preimage in $\KTh_0(R_0[[H]])$ consisting of a linear combination of classes of finitely generated $R_0[\Delta]$-modules which are free as $R_0$-modules. Hence, there exist $R_0[[H]]$-modules which are finitely generated and free over $R_0$ and have non-trivial class in $\KTh_0(R_0[[H]])$.

The same argumentation still works for $R_0$ replaced by $R$, $k$ replaced by $R'\coloneqq R_0[[X_1,\dots,X_{n-1}]]$ and $\pi$ replaced by $X_n$. The lifting argument in the last step becomes a bit easier. If $T$ is a $R'[[H]]$-module which is finitely generated and free over $R'$, then $T'\coloneqq R\tensor_{R'}T$ is a $R[[H]]$-module which is finitely generated and free over $R$ and satisfies $R'\tensor_R T'\isomorph T$. This completes the induction step.
%
\end{proof}

However, the vanishing in $\KTh_0(R[[H]])$ is not sufficient. Here is an example. Assume that $G=<\tau,\gamma>\isomorph \Int_{\ell}\rtimes\Int_{\ell}$ with $\gamma^{-1}\tau\gamma=\tau^{1+\ell}$. Set $H\coloneqq<\tau>$ and consider the constant $\Int_\ell[[G]]$-module $\Int_{\ell}$. Clearly, $[\Int_{\ell}]=0$ in $\KTh_0(\Int_{\ell}[[H]])$ according to Lem.~\ref{lem:vanishing in K0H}. However, $[\Int_{\ell}]\neq 0$ in $\KTh_0(\Int_{\ell}[[G]],S)$. Indeed, the complex
$$
\Int_{\ell}[[G]]\xrightarrow{v\mapsto \left(v-v\tau^{1+\ell},v-v\left(\sum_{i=0}^{\ell}\tau^i\right)\gamma\right)}\Int_{\ell}[[G]]^2\xrightarrow{(v,w)\mapsto v-v\gamma-w+w\tau^{1+\ell}}\Int_{\ell}[[G]]
$$
is a projective resolution of $\Int_{\ell}$. 
Hence, the image of $[\Int_{\ell}]$ in
\[
\KTh_0(\Int_{\ell}[[\Gamma]],S)=\Int_{\ell}[[\Gamma]]_S^\times/\Int_{\ell}[[\Gamma]]^\times
\]
under the natural projection map is given by the class of
$$
\frac{1-(\ell+1)\gamma}{1-\gamma}\in \Int_{\ell}[[\Gamma]]_S^\times,
$$ which is not in $\Int_{\ell}[[\Gamma]]^\times$.

A sufficient criterion for the vanishing of the class $[T(-1)_{\Gal_{K_\infty}}]$ in the group $\KTh_0(R[[G]],S)$ is given in \cite[Prop. 4.3.17]{FK:CNCIT}. Here is another one, inspired by \cite[Prop. 4.2]{Zabradi:Pairings}.

\begin{prop}\label{prop:vanishing of small modules}
Let $G=H\rtimes \Gamma$ be an $\ell$-adic Lie group without elements of order $\ell$. Assume that there exists a closed normal subgroup $N\subset G$ such that
\begin{enumerate}
 \item $G/N$ has no elements of order $\ell$ and the centraliser of every element in $G/N$ of finite order has infinitely many elements,
 \item the image of $H$ in $G/N$ is open.
 \end{enumerate}
Let further $R$ be a commutative, local, regular adic $\Int_\ell$-algebra. Then the class of every $R[[G]]$-module which is finitely generated as $R$-module is zero in $\KTh_0(R[[G]],S)$.
\end{prop}
\begin{proof}
By assumption $(1)$ and Lem.~\ref{lem:vanishing in K0H} the constant $R[[G/N]]$-module $R$ has trivial class in $\KTh_0(R[[G/N]])$. Set $G'\coloneqq G/N\times G/H$. Every $R[[G/N]]$-module may be considered as $R[[G']]$-module by letting $G/H$ act trivially. Thus, we see that $[R]=0$ in $\KTh_0(\cat{N}_{G/N}(R[[G']]))$, as well. By assumption $(2)$ every finitely generated $R[[G']]$-module which is finitely generated as $R[[G/N]]$-module may be considered via
$$
G\mto G',\qquad g\mapsto (gN,gH)
$$
as a finitely generated $R[[G]]$-module which is also finitely generated as $R[[H]]$-module. This induces an exact functor $\cat{N}_{G/N}(R[[G']])\mto\cat{N}_H(R[[G]])$ and hence, a homomorphism between the corresponding $\KTh$-groups. We conclude that $[R]=0$ also in $\KTh_0(\cat{N}_H(R[[G]]))$.

If $T$ is a $R[[G]]$-module which is finitely generated and free as $R$-module and $M$ is any module in $\cat{N}_H(R[[G]])$, we let $\Tor^R_i(T,M)$ denote the $i$-th left derived functor of the tensor product $T\tensor_{R}M$ with the diagonal action of $G$. Since $R$ is noetherian, any finitely generated, projective $R[[H]]$-module is flat as $R$-module and the same is true for $R[[G]]$. In particular, it does not matter if we compute $\Tor^R_i(T,M)$ in the category of finitely generated $R[[G]]$-modules or $R[[H]]$-modules or $R$-modules. Hence, $\Tor^R_i(T,M)$ is again in $\cat{N}_H(R[[G]])$ and it is finitely generated over $R$ if $M$ is a finitely generated $R$-module. We thus obtain an endomorphism
\[
\KTh_0(\cat{N}_H(R[[G]]))\mto\KTh_0(\cat{N}_H(R[[G]])),\qquad [M]\mapsto \sum_{i=0}^\infty (-1)^i[\Tor^R_i(T,M)],
\]
which maps $[R]$ to $[T]$. In particular, $[T]=0$.
\end{proof}

\begin{cor}\label{cor:vanishing of small modules II}
Let $G=H\rtimes\Gamma$ be an $\ell$-adic Lie group. Assume that
\begin{enumerate}
 \item $H$ is not virtually solvable and has no elements of order $\ell$,
 \item the Lie algebra $L(G)$ of $G$ decomposes as
       $$L(G)=L(H)\oplus V$$
       with $L(H)$ the Lie algebra of $H$ and some ideal $V$ of $L(G)$,
 \item $\ell-1>2\dim_{\Rat_{\ell}} L(H)$.
\end{enumerate}
Let further $R$ be a commutative, local, regular adic $\Int_\ell$-algebra. Then the class of every $R[[G]]$-module which is finitely generated as $R$-module is zero in $\KTh_0(R[[G]],S)$.
\end{cor}
\begin{proof}
By assumption $(2)$ there exist a characteristic open subgroup $H'\subset H$ and a closed subgroup $\Gamma'\isomorph\Int_{\ell}$ of the centraliser $Z_G(H')$ of $H'$ in $G$ such that $H'\cap \Gamma'=1$ and $H'\Gamma'$ is open in $G$. We may also assume that $H'$ is a uniformly powerful pro-$\ell$-group in the sense of \cite[Def.~4.1]{DixonSegal:AnalyticproP}. Set $d\coloneqq\dim_{\Rat_\ell}L(H)$. By \cite[Cor.~4.18]{DixonSegal:AnalyticproP} the automorphism group of $H'$ is isomorphic to a closed subgroup of $\GL_d(\Int_\ell)$, which does not have elements of order $\ell$ by assumption $(3)$. In particular, this is also true for $G/Z_G(H')$, as this group acts faithfully on $H'$ by conjugation. Since $\Gamma'\subset Z_G(H')$, the image of $H'$ must be open in $G/Z_G(H')$. This image is just the quotient of $H'$ by its centre $Z(H')$. Since $H'$ is not virtually solvable, the same must be true for $H'/Z(H')$ and therefore, also for $G/Z_G(H')$. We may thus apply Prop.~\ref{prop:vanishing of small modules} with $N=Z_G(H')$.
\end{proof}

\begin{rem}
In fact, one can replace condition $(1)$ in Prop.~\ref{prop:vanishing of small modules} by
\begin{enumerate}
\item[(1)'] $[R]=0$ in $\operatorname{G}_0(R[[G/N]])$.
\end{enumerate}
where $\operatorname{G}_0(R[[G/N]])$ is the Grothendieck group of all finitely generated $R[[G/N]]$-modules, dropping the assumption that $G/N$ has no elements of order $\ell$. Presumably, $(1)'$ is satisfied for all $G/N$ which are not virtually solvable. If this is true, then one can drop assumption $(3)$ in Cor.~\ref{cor:vanishing of small modules II}.
\end{rem}


\subsection{The main conjecture for Picard 1-motives}\label{ss:Picard-1-motives}

In this section, we will clarify the relation of the main conjectures in Section~\ref{ss:MCforladicReps} with the main conjecture for $\ell$-adic realisations of Picard 1-motives considered in \cite{GreitherPopescu:PicardOneMotives}.

We recall the notion of Picard 1-motives introduced by Deligne \cite{Deligne:HodgeIII}. For this, we need some more notation. Let $\Gm_X$ denote the \'etale sheaf defined by the group of units of a scheme $X$. For any closed immersion $i_Z\colon Z\mto X$ we let $\Gm_{X,Z}$ denote the kernel of the surjection
$$
\Gm_{X}\mto i_{Z*}\Gm_Z. 
$$
From now on, we assume that $X$ is a smooth and proper curve over a perfect field $k$ with function field $K(X)$ and that $Z$ is a finite closed subscheme. We let $j_{K(X)*}\Gm_{K(X)}$ denote the \'etale sheaf of invertible rational functions on $X$ and
\[
\sheaf{P}_Z\coloneqq \ker\left(j_{K(X)*}\Gm_{K(X)}\mto i_{Z*}i_{Z}^*(j_{K(X)*}\Gm_{K(X)}/\Gm_{X,Z})\right)
\]
its subsheaf of rational functions which are congruent to $1$ modulo the effective divisor of $X$ corresponding to $Z$ in the sense of \cite[Ch. III, \S 1]{Serre:AlgGrpsAndCF}. For any locally closed subscheme $Y$ of $X$ we let $\Div_Y$ denote the \'etale sheaf on $X$ of divisors with support on $Y$ and $\Div^0_Y$ denote the kernel of the degree map $\Div_Y\mto \Int$.

Consider the diagram
\begin{equation}\label{eqn:9-diagram}
\xymatrix@C-3pt{
&0\ar[d]&0\ar[d]&0\ar[d]&\\
0\ar[r]&\Gm_{X,Z}\ar[r]\ar[d]&\Gm_{X}\ar[r]\ar[d]&i_{Z*}\Gm_Z\ar[r]\ar[d]&0\\
0\ar[r]&\sheaf{P}_Z\ar[r]\ar[d]&j_{K(X)*}\Gm_{K(X)}\ar[r]\ar[d]&i_{Z*}i_{Z}^*(j_{K(X)*}\Gm_{K(X)}/\Gm_{X,Z}) \ar[r]\ar[d]&0\\
0\ar[r]&\Div_{X-Z}\ar[r]\ar[d]&\Div_{X}\ar[r]\ar[d]&\Div_{Z}\ar[r]\ar[d]&0\\
&0&0&0&}
\end{equation}
of \'etale sheaves on $X$. One checks easily by taking stalks that all the rows and columns are exact. By Hilbert 90 in the form of \cite[Prop.~III.4.9]{Milne:EtCohom}, we have
\[
\HF^1(X,i_{Z*}\Gm_Z)\isomorph\HF^1(Z,\Gm_Z)=\Pic(Z)=0.
\]
Hence, the third column is exact even in the category of presheaves. Clearly, this is also true for the third row. The weak approximation theorem for $K(X)$ implies the exactness of the second row in the category of presheaves. Hence,
$$
\HF^1(X,\sheaf{P}_Z)\subset\HF^1(X,j_{K(X)*}\Gm_{K(X)})
$$
and the latter group is zero by Hilbert 90 and the Leray spectral sequence. In particular, we have for any open dense subscheme $Y$ of $X$:
$$
\HF^1(Y,\Gm_{X,Z})=\Div_{X-Z}(Y)/\set{\operatorname{div}(f)\given f\in \sheaf{P}_Z(Y)}\eqqcolon\Pic(Y,Z\cap Y).
$$
This group is usually called the Picard group of $Y$ relative to the effective divisor corresponding to $Z\cap Y$. If $Y=X$ and $k$ is finite, then it is also known as the ray class group for the modulus $Z$. We let $\Pic^0(X,Z)\subset\Pic(X,Z)$ denote the subgroup of elements of degree $0$. If $k$ is algebraically closed, then it can be identified with $k$-valued points of the generalised Jacobian variety of $X$ with respect to $Z$ \cite[Ch.~V, Thm.~1]{Serre:AlgGrpsAndCF}.

We will now assume that $k=\algc{\FF}$. As before, we let $p>0$ denote the characteristic of $\FF$. Recall from \cite[Ex. III.1.9.(c)]{Milne:EtCohom} that an \'etale sheaf $\sheaf{F}$ on $X$ is flabby if $\HF^i(U,\sheaf{F})=0$ for all $i>0$ and all \'etale schemes $U$ of finite type over $X$.

\begin{lem}\label{lem:flabby resolution of relative Gm}
Let $X/\algc{\FF}$ be a smooth and proper curve over $\algc{\FF}$ and $Z\subset X$ be a finite closed subscheme. The complex of \'etale sheaves
\[
 \sheaf{P}_Z\mto \Div_{X-Z}
\]
is a flabby resolution of $\Gm_{X,Z}$ on $X$.
\end{lem}
\begin{proof}
Since $Z$ is a scheme of finite type of dimension $0$ over the separably closed field $\algc{\FF}$, the sheaves $i_{Z*}{\Gm}_Z$ and $\Div_Z$ are flabby. Hence, $j_{K(X)*}\Gm_{K(X)}/\sheaf{P}_Z$ is also flabby. As $\Div_{X-Z}$ and $\Div_{X}$ are direct sums of sheaves concentrated on closed points, they are flabby, as well. The sheaf $j_{K(X)*}\Gm_{K(X)}$ is flabby by \cite[Ex.~III.2.22.(d)]{Milne:EtCohom}. As the second row of the diagram~\eqref{eqn:9-diagram} is exact in the category of presheaves, $\sheaf{P}_Z$ must also be flabby.
\end{proof}

Consider two reduced closed subschemes $Z_1$ and $Z_2$ of $X$ with empty intersection. The Picard $1$-motive for $Z_1$ and $Z_2$ is defined to be the complex of abelian groups
$$
\Motive_{Z_1,Z_2}\colon \Div^0_{Z_1}(X)\mto \Pic^0(X,Z_2)
$$
concentrated in degrees $0$ and $1$ \cite[Def.~2.3]{GreitherPopescu:PicardOneMotives}. Its group of $n$-torsion points for a number $n>0$ is given by
$$
\Motive_{Z_1,Z_2,n}\coloneqq\HF^0(\Motive_{Z_1,Z_2}\Ltensor_\Int \Int/(n))
$$
and its $\ell$-adic Tate module is given by
$$
\Tate_{\ell}\Motive_{Z_1,Z_2}\coloneqq\varprojlim_{k>0}\Motive_{Z_1,Z_2,\ell^k}
$$
\cite[\S 10.1.5]{Deligne:HodgeIII}.

\begin{lem}\label{lem:n-torsion of 1-motives}
We have for all numbers $n>0$
$$
\Motive_{Z_1,Z_2,n}=\HF^0(\RDer\Sect(X-Z_1,\Gm_{X,Z_2})\Ltensor_{\Int}\Int/(n))
$$
where $\RDer\Sect(X-Z_1,\Gm_{X,Z_2})$ denotes the total derived section functor and $\Ltensor_{\Int}$ denotes the total derived tensor product in the derived category of abelian groups.
\end{lem}
\begin{proof}
Consider the complexes
\begin{align*}
\cmplx{A}&\colon \Div_{Z_1}(X)\mto \Pic(X,Z_2),\\
\cmplx{B}&\colon \Div_{Z_1}(X)\oplus \sheaf{P}_{Z_2}(X)\mto \Div_{X-Z_2}(X),
\end{align*}
and
$$
\cmplx{E}\coloneqq\begin{cases}
          \Int[-1]&\text{if $Z_1=\emptyset$,}\\
          0&\text{else,}
          \end{cases}
\qquad
\cmplx{F}\coloneqq\begin{cases}
          \algc{\FF}^\times&\text{if $Z_2=\emptyset$,}\\
          0&\text{else.}
          \end{cases}
$$
We obtain two obvious distinguished triangles
\begin{align*}
&\Motive_{Z_1,Z_2}\mto\cmplx{A}\mto\cmplx{E},
&\cmplx{F}\mto\cmplx{B}\mto\cmplx{A}.
\end{align*}
Moreover, the obvious map from $\cmplx{B}$ to the complex
$$
\sheaf{P}_{Z_2}(X-Z_1)\mto\Div_{X-Z_2}(X-Z_1)
$$
is a quasi-isomorphism. The latter complex may be identified with the complex $\RDer\Sect(X-Z_1,\Gm_{X,Z_2})$. For this, we note that
$$
\sheaf{P}_{Z_2}\mto\Div_{X-Z_2}
$$
is a flabby resolution of $\Gm_{X,Z_2}$ by Lem.~\ref{lem:flabby resolution of relative Gm}.

Since
$$
\HF^{-1}(\cmplx{E}\Ltensor_{\Int}\Int/(n))=\HF^{0}(\cmplx{E}\Ltensor_{\Int}\Int/(n))=
\HF^{0}(\cmplx{F}\Ltensor_{\Int}\Int/(n))=\HF^{1}(\cmplx{F}\Ltensor_{\Int}\Int/(n))=0,
$$
the proof is complete.
\end{proof}

\begin{lem}\label{lem:n-torsion 1-motive prime to p}
If $p\nmid n$, then
$$
\RDer\Sect(X-Z_1,\Gm_{X,Z_2})\Ltensor_{\Int}\Int/(n)\isomorph\RDer\Sect(X-Z_1,j_{(X-Z_2)!}\mu_n)[1]
$$
with $\mu_n$ the sheaf of $n$-th roots of unity on $X-Z_2$. In particular,
$$
\Motive_{Z_1,Z_2,n}\isomorph\HF^1(X-Z_1,j_{(X-Z_2)!}\mu_n)\isomorph\dual{\HF^1(X-Z_2,j_{(X-Z_1)!}\Int/(n))}
$$
\end{lem}
\begin{proof}
The first statement follows from the Kummer sequences for $\Gm_X$ and $\Gm_{Z_2}$ and the exactness of the sequence
$$
0\mto j_{(X-Z_2)!}\mu_n\mto\mu_n\mto i_{Z_2}^*\mu_n\mto 0.
$$
The second statement follows from Lem.~\ref{lem:n-torsion of 1-motives} and Cor.~\ref{cor:Poincare duality with Pontryagin dual}.
\end{proof}

\begin{lem}\label{lem:n-torsion 1-motive p primary}
For all numbers $k>0$ the canonical morphism
$$
\RDer\Sect(X-Z_1,\Gm_{X,Z_2})\Ltensor_{\Int}\Int/(p^k)\mto\RDer\Sect(X-Z_1,\Gm_{X})\Ltensor_{\Int}\Int/(p^k)\isomorph\RDer\Sect(X-Z_1,\nu_k^1)
$$
is an isomorphism. Here, $\nu_k^1\coloneqq\operatorname{W}_k\Omega^1_{X,\mathrm{log}}$ is the logarithmic de~Rham-Witt sheaf on $X$. In particular,
$$
\Motive_{Z_1,Z_2,p^k}\isomorph\HF^0(X-Z_1,\nu_k^1)\isomorph\dual{\HF^1(X,j_{(X-Z_1)!}\Int/(p^k))}.
$$
\end{lem}
\begin{proof}
Since we assume $Z_2$ to be reduced, we have
$$
\RDer\Sect(Z_2,\Gm_{Z_2})\Ltensor_{\Int}\Int/(p^k)\isomorph 0.
$$
This explains the first isomorphism in the first part of the statement. For the second isomorphism we may use \cite[Prop.~2.2]{Geisser:Duality} together with the identifications
$$
\Int^c_X\isomorph \Int_X(1)[2]\isomorph \Gm_X[1]
$$
in the notation of \emph{loc.\,cit.\,}. The duality statement
$$
\HF^0(X-Z_1,\nu_k^1)\isomorph\dual{\HF^1(X,j_{(X-Z_1)!}\Int/(p^k))}
$$
can be deduced from \cite[Thm.~4.1]{Geisser:Duality}:
\begin{align*}
\RDer\Sect(X-Z_1,\nu_k^1)&\isomorph\RDer\Sect(X-Z_1,\Int^c_X)\Ltensor_{\Int}\Int/(p^k)[-1]\\
                         &\isomorph\RDer\Hom_{X-Z_1}(\Int/(p^k),\Int^c_X)\\
                         &\isomorph\RDer\Hom(\RDer\Sect(X,j_{(X-Z_1)!}\Int/(p^k)),\Int)\\
                         &\isomorph\dual{\RDer\Sect(X,j_{(X-Z_1)!}\Int/(p^k))}[-1].
\end{align*}
\end{proof}


As before, we consider an admissible extension $K_\infty/K$ with Galois group $G=H\rtimes \Gamma$ and two open dense subschemes $V$ and $W$ of the proper smooth curve $C$ over $\FF$ with function field $K$ such that $C=V\cup W$. We will assume that $H=\Gal(K_\infty/K_\cyc)$ is finite. Set $U\coloneqq V\cap W$ and $\Upsilon\coloneqq\Gal(\algc{\FF}K_{\infty}/K_\infty)$. Recall that $\Upsilon$ is of order prime to $\ell$. As $\Int_{\ell}$ is of global dimension $1$, Example~\ref{exmpl:ramification types} implies that the finitely ramified representations $\Int_\ell$ and  $\Int_{\ell}(1)$ (if $\ell\neq p$) of $\Gal_K$ over $\Int_{\ell}$ have projective stalks over $K_\infty$ in all closed points of $C$.

We recall from Def.~\ref{defn:LambdaH-dual} that for a $\Int_\ell[[G]]$-module $M$ which is finitely generated and free as $\Int_{\ell}$-module, $\cg{M}^{\mdual_{\Int_\ell}}$ denotes the $\Int_\ell$-dual considered as left $\Int_\ell[[G]]$-module.

\begin{prop}\label{prop:Tate modules of Picard-1-motives}
Assume that $H$ is finite.
\begin{enumerate}
 \item If $\ell\neq p$, then $\HF^2(W,j_{U!}j_{K*}\Int_\ell[[G]]^\sharp)$ is a finitely generated free $\Int_\ell$-module and
       \begin{align*}
       (\Tate_\ell\Motive_{\Sigma_{V_{\algc{\FF}K_\infty}},\Sigma_{W_{\algc{\FF}K_\infty}}})^{\Upsilon}&\isomorph \HF^2(V,j_{U!}j_{K*}\Int_\ell[[G]]^\sharp(1))\\
       &\isomorph\cg{\HF^2(W,j_{U!}j_{K*}\Int_\ell[[G]]^\sharp)^{\mdual_{\Int_{\ell}}}};
       \end{align*}
 \item if $\ell=p$, then $\HF^2(C,j_{U!}j_{K*}\Int_\ell[[G]]^\sharp)$ is a finitely generated free $\Int_p$-module and
       $$
       (\Tate_p\Motive_{\Sigma_{V_{\algc{\FF}K_\infty}},\Sigma_{W_{\algc{\FF}K_\infty}}})^{\Upsilon}\isomorph\cg{\HF^2(C,j_{V!}j_{K*}\Int_p[[G]]^\sharp)^{\mdual_{\Int_{\ell}}}}.
       $$
\end{enumerate}
\end{prop}
\begin{proof}
We first assume that $U\neq W$ (and  $W=C$ if $\ell=p$). Then the complex $\RDer\Sect(W_{K_\infty},j_{U_{K_\infty}!}j_{K_\infty*}\Int_{\ell})$ is quasi-isomorphic to a strictly perfect complex of $\Int_{\ell}$-modules $\cmplx{P}$ with $P^i=0$ for $i\notin\set{1,2}$. In particular, $\HF^1(\cmplx{P})$ is a submodule of the finitely generated and free $\Int_{\ell}$-module $P^1$ and therefore, finitely generated and free, as well. If $U=W$ we let $j'\colon W'\mto W_{K_\infty}$ be the inclusion of the complement of a single closed point. We then see that
$$
\HF^1(W_{K_\infty},j_{K_\infty*}\Int_{\ell})=\HF^1(W_{K_\infty},j'_!j_{K_\infty*}\Int_{\ell})
$$
is still a finitely generated and free $\Int_{\ell}$-module. By Lem.~\ref{lem:HS argument}, this module is isomorphic to $\HF^2(W,j_{U!}j_{K*}\Int_\ell[[G]]^\sharp)$. Then we apply  Lem.~\ref{lem:n-torsion 1-motive prime to p} and Lem.~\ref{lem:n-torsion 1-motive p primary}.
\end{proof}

\begin{rem}
Note that the image of $\Upsilon$ in $\Aut_{\Int_{\ell}}(\Tate_\ell\Motive_{\Sigma_{V_{\algc{\FF}K_\infty}},\Sigma_{W_{\algc{\FF}K_\infty}}})$ is finite. Hence, we can always choose the admissible extension $K_\infty/K$ large enough such that
$$
(\Tate_\ell\Motive_{\Sigma_{V_{\algc{\FF}K_\infty}},\Sigma_{W_{\algc{\FF}K_\infty}}})^{\Upsilon}
=\Tate_\ell\Motive_{\Sigma_{V_{\algc{\FF}K_\infty}},\Sigma_{W_{\algc{\FF}K_\infty}}}.
$$
\end{rem}

Recall that $q$ denotes the number of elements in $\FF$ and that $\gamma$ is the image of the geometric Frobenius in $\Gamma$. The following two corollaries are a non-commutative generalisation of Greither's and Popescu's main conjecture for Picard-1-motives \cite[Cor. 4.13]{GreitherPopescu:PicardOneMotives}.

\begin{cor}\label{cor:picard motives Tate twist 1}
Assume that $\ell\neq p$, that $H$ is finite, that both $\Sigma_V$ and $\Sigma_W$ are non-empty, and that $K_\infty/K$ has ramification prime to $\ell$ over $U$. Then:
\begin{enumerate}
\item The $\Int_\ell[[G]]$-module $(\Tate_\ell\Motive_{\Sigma_{V_{\algc{\FF}K_\infty}},\Sigma_{W_{\algc{\FF}K_\infty}}})^{\Upsilon}$ is finitely generated and projective over $\Int_{\ell}[[H]]$. In particular, it has a well-defined class in the Grothen\-dieck group $\KTh_0(\Int_{\ell}[[G]],S)$.
\item We have
$$
\bh\ncL_{K_\infty/K,\Sigma_W,\Sigma_V}(\Int_{\ell}(1))=
-\left[(\Tate_\ell\Motive_{\Sigma_{V_{\algc{\FF}K_\infty}},\Sigma_{W_{\algc{\FF}K_\infty}}})^{\Upsilon}\right]
$$
in $\KTh_0(\Int_{\ell}[[G]],S)$.
\item Let $\Lambda$ be a commutative adic $\Int_\ell$-algebra and $\rho$ a $\Lambda$-$\Int_\ell[[G]]$-bimodule which is finitely generated and projective as $\Lambda$-module. Then
    $$
    \eval_\rho(\ncL_{K_\infty/K,\Sigma_W,\Sigma_V}(\Int_{\ell}(1)))=L_{\Sigma_W,\Sigma_V}(\rho^\sharp,q^{-1}\gamma^{-1})
    $$
\end{enumerate}
\end{cor}
\begin{proof}
This follows from Thm.~\ref{thm:MCforRep l p different} with $T=\Int_{\ell}(1)$ together with Cor.~\ref{cor:projectivity ell neq pee} and Prop.~\ref{prop:Tate modules of Picard-1-motives}.
\end{proof}

\begin{cor}\label{cor:Picard motives Tate twist 0}
Assume that $H$ is finite and that $\Sigma_W$ is not empty.  If $\ell\neq p$ we also assume that $\Sigma_V$ is not empty and that $K_\infty/K$ has ramification prime to $\ell$ over $U$. If $\ell=p$ we assume that $K_\infty/K$ has ramification prime to $p$ over $V$. Then:
\begin{enumerate}
\item The $\Int_\ell[[G]]$-module $\cg{((\Tate_\ell\Motive_{\Sigma_{W_{\algc{\FF}K_\infty}},\Sigma_{V_{\algc{\FF}K_\infty}}})^{\Upsilon})^{\mdual_{\Int_{\ell}}}}$ is finitely generated and projective over $\Int_{\ell}[[H]]$. In particular, it has a well-defined class in the Grothendieck group $\KTh_0(\Int_{\ell}[[G]],S)$.
\item We have
$$
\bh\ncL_{K_\infty/K,\Sigma_W,\Sigma_V}(\Int_{\ell})=
-\left[\cg{((\Tate_\ell\Motive_{\Sigma_{W_{\algc{\FF}K_\infty}},\Sigma_{V_{\algc{\FF}K_\infty}}})^{\Upsilon})^{\mdual_{\Int_{\ell}}}}\right]
$$
in $\KTh_0(\Int_{\ell}[[G]],S)$.
\item Let $\Lambda$ be a commutative adic $\Int_\ell$-algebra and $\rho$ a $\Lambda$-$\Int_{\ell}[[G]]$-bimodule which is finitely generated and projective as $\Lambda$-module. Then
    \begin{align*}
    \eval_\rho(\ncL_{K_\infty/K,\Sigma_W,\Sigma_V}(\Int_{\ell}))&=L_{\Sigma_W,\Sigma_V}(\rho^\sharp,\gamma^{-1}).
    \end{align*}
\end{enumerate}
\end{cor}
\begin{proof}
This follows from Thm.~\ref{thm:MCforRep l p different} and Thm.~\ref{thm:MCforRep l equal p myLfunc} with $T=\Int_{\ell}$ together with Cor.~\ref{cor:projectivity ell eq pee} and Prop.~\ref{prop:Tate modules of Picard-1-motives}.
\end{proof}

\subsection{The main conjecture for function fields}

In this section, we will deduce a non-commutative function field analogue of the most classical formulation of the Iwasawa Main Conjecture. We retain the notation of the previous sections. In particular, we fix an admissible $\ell$-adic Lie extension $K_\infty/K$ with Galois group $G=H\rtimes\Gamma$. Different from Section~\ref{ss:Picard-1-motives}, we do no longer assume the group $H$ to be finite, but we will assume that $\Sigma_V=\emptyset$ and hence, $V=C$, $W=U$, and $\Sigma_W=C-W$. Further, we will write $\Sigma_0$ for the closed subset of $W$ where $K_\infty/K$ has non-torsion ramification.

\begin{cor}\label{cor:classical mc}
Assume that $\ell\neq p$ and that $G$ does not contain any element of order $\ell$. Let $M$ be the maximal abelian $\ell$-extension of $K_\infty$ which is unramified outside $\Sigma_W$. Then
\begin{enumerate}
\item $\Gal(M/K_\infty)$ is in $\cat{N}_H(\Int_{\ell}[[G]])$ and
      \begin{align*}
      \bh \ncL_{K_\infty/K,\Sigma_W\cup \Sigma_0,\emptyset}(\Int_{\ell}(1))=&-[\Gal(M/K_\infty)]+[\Int_{\ell}]\\
      &+
      \begin{cases}
      [\Int_{\ell}(1)]&\text{if $\Sigma_W=\emptyset$, $H$ is finite, and $\mu_{\ell}\subset K_{\infty}$}\\
      0&\text{else.}
      \end{cases}
      \end{align*}
      in $\KTh_0(\Int_{\ell}[[G]],S)$
\item Let $\Lambda$ be a commutative adic $\Int_\ell$-algebra and  $\rho$ be a $\Lambda$-$\Int_\ell[[G]]$-bimodule which is finitely generated and projective as $\Lambda$-module. Then
    \begin{align*}
    \eval_\rho(\ncL_{K_\infty/K,\Sigma_W\cup\Sigma_0,\emptyset}(\Int_{\ell}(1)))&=L_{\Sigma_W\cup\Sigma_0,\emptyset}(\rho,q^{-1}\gamma^{-1}).
    \end{align*}
\end{enumerate}
\end{cor}
\begin{proof}
From Lem.~\ref{lem:calculation of cohomology} and from the equality
$$
\HF^1(W_{K_\infty},\Rat_\ell/\Int_\ell)=\dual{\Gal(M/K_\infty)}.
$$
we deduce
$$
\HF^2(C,j_{W!}j_{K*}\Int_{\ell}[[G]]^\sharp(1))=\X_{K_\infty/K,\Sigma_W}(\Int_{\ell}(1))
               =\Gal(M/K_\infty).
$$
We then apply Thm.~\ref{thm:MCforRep l p different} and Cor.~\ref{cor:mc for selmer groups}. Finally, we remark that $\Int_{\ell}(1)^{\Gal_{K_\infty}}=0$ if $K_\infty$ does not contain any $\ell$-th root of unity. If $K_\infty$ does contain an $\ell$-th root of unity, then it contains also all $\ell^n$-th roots of unity for any $n$, and therefore, $\Int_{\ell}(1)^{\Gal_{K_\infty}}=\Int_{\ell}(1)$ in this case.
\end{proof}

If $G$ does contain elements of order $\ell$, then Thm.~\ref{thm:MCforRep l p different} applied to $\Int_{\ell}(1)$ is still a sensible main conjecture if we assume that $K_\infty/K$ has ramification prime to $\ell$ over $W$; however, we can no longer replace the class of the complex
$$
[\RDer\Sect(C,j_{W!}j_{K*}(\Lambda[[G]]^{\sharp}\tensor_{\Lambda} \Int_\ell(1)))]=-\bh \ncL_{K_\infty/K,\Sigma_W,\emptyset}(\Int_{\ell}(1))
$$
by the classes of its cohomology modules. One may also apply Thm.~\ref{thm:MCforRep l p different} and Thm.~\ref{thm:MCforRep l equal p myLfunc} to $\Int_{\ell}$ resulting in a main conjecture for every $\ell$. Main conjectures of this type have already been discussed in \cite{Burns:MCinGIwTh+RelConj}.

\subsection{The main conjecture for abelian varieties}

In this section we let $A$ be an abelian variety over $\Spec K$.  We continue to assume $\ell\neq p$ and that $\Sigma_V=\emptyset$. Our aim is to deduce a precise function field analogue of the $\GL_2$ main conjecture in \cite{CFKSV}.

Let $\Val$ be the valuation ring of a finite extension of $\Rat_\ell$ and $\rho$ a finitely ramified representation of $\Gal_K$ over $\Val$. The  $\Sigma_W$-truncated $L$-function of $A$ twisted by $\rho$ is given by
$$
L_{\Sigma_W}(A,\rho,t)\coloneqq\prod_{v\in W^0}\det(1-(\Frob_v t)^{\deg(v)}\colon (\HF^1(A\times_{\Spec K}\Spec \algc{K},\Rat_{\ell})\tensor_{\Int_{\ell}}\rho)^{\inertia_v})^{-1}.
$$
If $\rho$ is an Artin representation of $\Gal_K$, then $L_{\Sigma_W}(A,\rho,q^{-s})$ is the Hasse-Weil $L$-function of $A$ twisted by $\rho$.

We will write $\check{A}$ for the dual abelian variety,
$$
A(\algc{K})_n\coloneqq\ker A(\algc{K})\xrightarrow{n}A(\algc{K})
$$
for the group of $n$-torsion points and
$$
\Tate_{\ell}A\coloneqq\varprojlim_{k}A(\algc{K})_{\ell^k}
$$
for the $\ell$-adic Tate module of $A$. It is well known that $\Tate_{\ell} A$ is a finitely ramified representation of $\Gal_K$ over $\Int_{\ell}$. Moreover, the argument of \cite[\S 1]{Schn:VBSDGFK} shows that for any closed point $v\in C$
$$
(j_{K*}(\Tate_{\ell}\check{A}(-1)\tensor_{\Int_\ell}\rho))_v\tensor_{\Int_{\ell}}\Rat_{\ell}\isomorph (\HF^1(A\times_{\Spec K}\Spec \algc{K},\Rat_{\ell})\tensor_{\Int_{\ell}}\rho)^{\inertia_v}
$$
such that
$$
L_{\Sigma_W}(A,\rho,q^{-1}t)=L_{\Sigma_W,\emptyset}(\Tate_{\ell}\check{A}\tensor_{\Int_{\ell}}\rho,t).
$$

As an immediate consequence of Thm.~\ref{thm:MCforRep l p different} we obtain:

\begin{cor}\label{cor:MCforAbVars}
Let $K_\infty/K$ be an admissible extension with $\ell\neq p$. Assume that $K_\infty/K$ has ramification prime to $\ell$ over $W$. Then:
\begin{enumerate}
\item We have
$$
\bh \ncL_{K_{\infty}/K,\Sigma_W,\emptyset}(\Tate_{\ell}\check{A})=-[\RDer\Sect(C,j_{W!}j_{K*}(\Int_\ell[[G]]^{\sharp}\tensor_{\Int_\ell} \Tate_{\ell}\check{A}))]
$$
in $\KTh_0(\Int_{\ell}[[G]],S)$.
\item Let $\Val$ be the ring of integers of a finite extension field of $\Rat_\ell$ and $\rho$ a $\Val$-$\Int_{\ell}[[G]]$-bimodule which is finitely generated and projective as $\Val$-module. Then
    $$
    \eval_\rho(\ncL_{K_\infty/K,\Sigma_W,\emptyset}(\Tate_{\ell}\check{A}))=L_{\Sigma_W}(A,\rho^\sharp,q^{-1}\gamma^{-1}).
    $$
\end{enumerate}
\end{cor}

For any extension $L/K$ inside $\algc{K}$ we let
$$
\Sel_{\Sigma_W}(L,A)\coloneqq\varinjlim_{k}\ker \HF^1(\Gal_L,A(\algc{K})_{\ell^k})\mto \bigoplus_{v\in W^0_L}\HF^1(\decomp_v,A(\algc{K}))
$$
be the $\Sigma_W$-truncated Selmer group of $A$ over $L$. Here, $\decomp_v$ denotes a choice of a decomposition subgroup of $v$ inside $\Gal_L$ as at the beginning of Section~\ref{sec:Selmer Complexes}.

\begin{lem}
For every admissible $\ell$-adic Lie extension $K_\infty/K$ with $\ell\neq p$ we have
$$
\Sel_{\Sigma_W}(K_\infty,A)\isomorph\Sel_{\Sigma_W}(K_\infty,\dual{\Tate_{\ell}(\check{A})}(1))
$$
\end{lem}
\begin{proof}
Let $L$ be an extension of $K$ and let $L_v$ be the completion of $L$ at $v\in W_L$. According to Greenberg's approximation theorem we have
$$
\HF^1(\decomp_v,A(\algc{K}))=\HF^1(\Gal_{L_v},A(\algc{L_v}))
$$
\cite[Rem.~I.3.10]{Milne:ADT} for all finite extensions $L/K$. Since the points of the formal group of $A$ form an open pro-$p$-subgroup of $A(L_v)$ we conclude from the Kummer sequence that
$$
\Sel_{\Sigma_W}(L,A)=\varinjlim_{k}\ker \HF^1(\Gal_L,A(\algc{K})_{\ell^k})\mto \bigoplus_{v\in W^0_L}\HF^1(\decomp_v,A(\algc{K})_{\ell^k})
$$
for all extensions $L/K$ inside $\algc{K}$.
If $\FF_\cyc\subset L$, then $\decomp_v/\inertia_v$ is a profinite group of order prime to $\ell$ and the Hochschild-Serre spectral sequence shows that
$$
\HF^1(\decomp_v,A(\algc{K})_{\ell^k})\mto \HF^1(\inertia_v,A(\algc{K})_{\ell^k})
$$
is an injection. Furthermore,
$$
\dual{\Tate_{\ell}(\check{A})}(1)=\varinjlim_{k} A(\algc{K})_{\ell^k}
$$
\cite[\S 1]{Schn:VBSDGFK} such that indeed $\Sel_{\Sigma_W}(K_\infty,A)\isomorph\Sel_{\Sigma_W}(K_\infty,\dual{\Tate_{\ell}(\check{A})}(1))$.
\end{proof}

In particular, we deduce the following function field analogue of the $\GL_2$ main conjecture of \cite{CFKSV} as a special case of Cor.~\ref{cor:mc for selmer groups}.

\begin{cor}
Let $\ell\neq p$, $K_\infty/K$ be an admissible $\ell$-adic Lie extension with Galois group $G$, and $A$ an abelian variety over $\Spec K$. We assume that $G$ does not contain any element of order $\ell$ and write $\Sigma_0$ for the set of points in $W$ in which $K_\infty/K$ has non-torsion ramification. Then $\dual{\Sel_{\Sigma_W}(K_\infty,A)}$ is in $\cat{N}_H(\Int_{\ell}[[G]])$ and
\begin{align*}
\bh \ncL_{K_\infty/K,\Sigma_W\cup\Sigma_0,\emptyset}(\Tate_\ell(\check{A}))&=-[\dual{\Sel_{\Sigma_W}(K_\infty,A)}]+[\Tate_\ell(\check{A})(-1)_{\Gal_{K_\infty}}]\\
      &+
      \begin{cases}
      [\Tate_\ell(\check{A})^{\Gal_{K_\infty}}]&\text{if $\Sigma_W=\emptyset$ and $H$ is finite,}\\
      0&\text{else.}
      \end{cases}
\end{align*}
in $\KTh_0(\Int_{\ell}[[G]],S)$.
\end{cor}

The terms $[\Tate_\ell(\check{A})(-1)_{\Gal_{K_\infty}}]$ and  $[\Tate_\ell(\check{A})^{\Gal_{K_\infty}}]$ disappear in the following situation. Recall that by an old result of Grothendieck \cite[Thm. 1.1]{Oort:TheIsogenyClass}, an abelian variety over $K$ is of $CM$-type over a $\algc{K}$ if and only if it is isogenous to an abelian variety over a finite field. Moreover, this is the case if  the image of $\Gal_{K_\cyc}$ in the automorphism group of $\Tate_\ell(A)$ is finite \cite[Last step]{Oort:TheIsogenyClass}.

\begin{prop}\label{prop:vanishing of the tate module class}
Let $A$ be an abelian variety over $\Spec K$ of dimension $g\geq 1$ which is not of $CM$-type over $\algc{K}$. Let $K_\infty$ be the extension of $K$ obtained by adjoining the coordinates of all $\ell^n$-torsion points of $A$. If $\ell>8g^2-1$, then $K_\infty/K$ is an admissible $\ell$-adic Lie extension, $\Gal(K_\infty/K)$ does not contain any element of order $\ell$ and
$$
\bh \ncL_{K_\infty/K,\Sigma_W,\emptyset}(\Tate_\ell(\check{A}))=-[\dual{\Sel_{\Sigma_W}(K_\infty,A)}].
$$
in $\KTh_0(\Int_{\ell}[[\Gal(K_\infty/K)]],S)$.
\end{prop}
\begin{proof}
It is well known that $\Gal(K_\infty/K)$ is the image of $\Gal_K$ in $\Aut_{\Int_{\ell}}(\Tate_\ell(\check{A}))$, that $\Tate_\ell(\check{A})$ is a free $\Int_{\ell}$-module of rank $2g$, and that $\Gal_K$ acts on the determinant of $\Tate_\ell(\check{A})$ via the cyclotomic character $\cycchar$. This shows that $K_\infty/K$ is an admissible $\ell$-adic Lie extension. Since $\ell-1>2g$, the group $\Aut_{\Int_{\ell}}(\Tate_\ell(\check{A}))$ does not contain any element of order $\ell$. By a result of Zarhin \cite[\S 4]{Zarhin:TorsionAbVars}, \cite[\S 6]{Zarhin:AbVarsOverFiniteChar}, the Lie algebra $L(G)$ of $G$ is the direct product
\[
L(G)=\mathfrak{g}^0\times \mathfrak{c}
\]
of a semi-simple Lie algebra $\mathfrak{g}^0$ of dimension less or equal to $4g^2-1$ over $\Rat_{\ell}$ and a commutative Lie algebra $\mathfrak{c}$ of dimension $1$. Since any finite extension of $K$ has only one $\Int_\ell$-extension, $\mathfrak{g}^0$ necessarily coincides with $L(H)$. Since $A$ is not of $CM$-type over $\algc{K}$, $H$ is not finite and hence, $L(H)$ is non-trivial. In particular,
$$
[\Tate_{\ell}(\check{A}(-1))]=0
$$
in $\KTh_0(\Int_{\ell}[[\Gal(K_\infty/K)]],S)$ by Cor.~\ref{cor:vanishing of small modules II}.
\end{proof}

\begin{rem}
With $K_\infty$ as in Prop.~\ref{prop:vanishing of the tate module class}, assume that $\ell>2g-1$, such that $G$ has no elements of order $\ell$.
\begin{enumerate}
\item By the above result of Zarhin, one can always find a finite extension $K'/K$ inside $K_\infty/K$ such that
\[
\Gal(K_\infty/K')=\Gal(K_\infty/K'_\cyc)\times\Gal(K'_\cyc/K').
\]
Hence,
$$
[\Tate_{\ell}(\check{A}(-1))]=0
$$
in $\KTh_0(\Int_{\ell}[[\Gal(K_\infty/K')]],S)$ by Prop.~\ref{prop:vanishing of small modules} applied to $N=\Gal(K'_\cyc/K')$.
\item One may also try to apply the criterion of \cite[Prop. 4.3.17]{FK:CNCIT} to the module $\Tate_{\ell}(\check{A}(-1))$. However, one of the requirements is that $G$ has infinite intersection with the subgroup
\[
\Int_{\ell}^\times\id\subset\Aut_{\Int_{\ell}}(\Tate_\ell(\check{A})).
\]
Different from the number field case, this condition is not always satisfied for abelian varieties over $K$. Zarhin constructs in \cite{Zarhin:AbVarsWOHomoth} for every odd $g>1$ examples of abelian varieties of dimension $g$ which are not of $CM$-type and such that $G$ has finite intersection with $\Int_{\ell}^\times\id$ independent of the choice of $\ell$.
\item If $g=1$, then one can always take $K'=K$. Indeed, $\Gal(K_\infty/K)$ must be open in $\Aut_{\Int_{\ell}}(\Tate_\ell(\check{A}))=\GL_2(\Int_\ell)$ and the intersection of $\Gal(K_\infty/K_\cyc )$ with $\SL_2(\Int_\ell)$ is open in $\SL_2(\Int_\ell)$. Otherwise, $\Gal(K_\infty/K)$ would contain a commutative open subgroup by the above result of Zarhin, which is not possible since $A$ is not of $CM$-type over $\algc{K}$ (This was also observed in the thesis \cite{Sechi:IMCoverGlobalFunctionFields}, using a different argument). By the assumption on $\ell$ we may write
$$
\GL_2(\Int_\ell)=H'\times \Int_\ell
$$
with $H'$ not virtually solvable. So we may apply Prop.~\ref{prop:vanishing of small modules}.
\end{enumerate}
\end{rem}

\bibliographystyle{amsalpha}
\bibliography{Literature}
\end{document}